\documentclass[10pt]{amsart}
\usepackage{graphicx}
\usepackage{amsmath,amssymb,amsthm}
\usepackage{url}
\usepackage{a4wide}

\newtheorem{theorem}{Theorem}[section]
\newtheorem{corollary}[theorem]{Corollary}
\newtheorem{remark}[theorem]{Remark}
\newtheorem{proposition}[theorem]{Proposition}
\newtheorem{lemma}[theorem]{Lemma}
\newtheorem{definition}[theorem]{Definition}

\newcommand{\ze}{\zeta}
\newcommand{\PP}{\mathcal{D}_{N}}
\newcommand{\PPtwo}{\mathcal{D}_{2N}}
\newcommand{\PN}{\mathcal{P}_{2N}}
\newcommand{\NN}{\mathcal{N}}
\newcommand{\JJ}{\mathcal{J}}
\newcommand{\al}{\alpha} \newcommand{\be}{\beta}
\newcommand{\ga}{\gamma} \newcommand{\de}{\delta}
\newcommand{\De}{\Delta}

\newcommand{\si}{\sigma} 
\newcommand{\e}{\varepsilon}

\newcommand{\R}{{\mathbb R}}
\newcommand{\N}{{\mathbb N}}
\newcommand{\Z}{{\mathbb Z}}

\newcommand{\Id}{\mathrm{Id}}

\newcommand{\abs}[1]{{\left\lvert#1\right\rvert}}
\newcommand{\norm}[1]{{\left\lVert#1\right\rVert}}

\newcommand{\defeq}{:=}
\newcommand{\defqe}{=:}
\newcommand{\ddx}{\frac{\mathrm{d}}{\mathrm{d}x}}
\renewcommand{\rho}{\varrho}
\renewcommand{\(}{\left(}
\renewcommand{\)}{\right)}
\renewcommand{\d}{\, \mathrm{d}}

\def\sp#1#2{\left\langle#1,#2\right\rangle}

\newcommand{\sign}{\mathrm{sign}}

\newcommand{\Jac}{\mathrm{Jac}}
\newcommand{\G}{{F}}
\newcommand{\PG}{{P_F}}
\newcommand{\E}[1]{{e^h(q_{#1})}}
\newcommand{\Ek}[1]{{e^h(q^k_{#1})}}
\newcommand{\DE}{{\nabla_N e^h}}

\newcommand{\Cbar}{{\Delta \Lbar}}
\newcommand{\Dbar}{{\bar D}}
\newcommand{\gapc}{{\ga_k^\mathrm{pc}}}
\newcommand{\gapl}{{\ga_k^\mathrm{pl}}}
\newcommand{\gapq}{{\ga_k^\mathrm{pq}}}
\newcommand{\gapco}{{\ga^\mathrm{pc}}}
\newcommand{\gaplo}{{\ga^\mathrm{pl}}}
\newcommand{\gapqo}{{\ga^\mathrm{q}}}

\newcommand{\dedx}{\tfrac{\mathrm{d}}{\mathrm{d} x}}
\newcommand{\dezdx}{\tfrac{\mathrm{d}^2}{\mathrm{d} x^2}}
\def\sse{{q+\tfrac12(\e,\De)}}

\def\seg#1#2{\left[#1,#2\right]}

\def\DD#1#2#3{D(#1,#2,#3)}
\def\DDbar#1#2#3{\bar D(#1,#2,#3)}

\def\Lbar{\bar L}

\let\nnu=\nu
\def\nu{{\nnu^i}}
\def\nnuu{{e_k}}
\let\a=\abs

\newenvironment{acknowledgments}{\subsection*{Acknowledgements}
\small}%

\begin{document}

\title[A convergent string method]{A convergent string method:
  Existence and approximation for the Hamiltonian boundary-value
  problem}

\author{Hartmut Schwetlick and Johannes Zimmer}
\thanks{
both Department of Mathematical Sciences, University of Bath,
    Bath BA2 7AY, United Kingdom, \texttt{\{schwetlick|zimmer\} at
      maths.bath.ac.uk}}

\date{\today}
\maketitle

\centerline{
Dedicated to Professor Armin Leutbecher on the occasion of his 80th birthday}

\begin{abstract}
  This article studies the existence of long-time solutions to the   Hamiltonian boundary value problem, and their consistent numerical   approximation. Such a boundary value problem is, for example, common   in Molecular Dynamics, where one aims at finding a dynamic   trajectory that joins a given initial state with a final one, with   the evolution being governed by classical (Hamiltonian) dynamics.   The setting considered here is sufficiently general so that long   time transition trajectories connecting two configurations can be   included, provided the total energy $E$ is chosen suitably. In   particular, the formulation presented here can be used to detect   transition paths between two stable basins and thus to prove the   existence of long-time trajectories. The starting point is the   formulation of the equation of motion of classical mechanics in the   framework of Jacobi's principle; a curve shortening procedure   inspired by Birkhoff's method is then developed to find geodesic   solutions.  This approach can be viewed as a string method.
\end{abstract}

\section{Introduction}
\label{sec:introduction}

The aim of this article is to study the existence and give a consistent approximation procedure of the boundary value problem for the conservative dynamical system
\begin{equation}
  \label{eq:hamildyn}
  \frac{d^2 q(t)}{\mathrm{d} t^2} = - \nabla V(q) ,
\end{equation}
where $V$ is a smooth potential on $Q$. We assume that $Q$ is an open subset of $\R^n$ as this is the relevant case for the applications we have in mind; extensions to a more general setting are possible but not discussed here.

For the boundary conditions, we write
\begin{equation}
\label{eq:boundarycond}
  q(0) = q_a \text{ and } q(T_0) = q_b
\end{equation}
with $q_a, q_b \in Q$ and $T_0>0$. Here, $T_0$ is part of the problem and has to be determined (however, the total energy $E$, defined as the sum of kinetic and potential energy, is fixed). The focus on the boundary-value problem is motivated by applications, as discussed below.

\subsection{Hamiltonian systems, rare events and path sampling}
\label{sec:Hamilt-syst-rare}

Equation~\eqref{eq:hamildyn} (furnished with various initial or boundary conditions) can be reformulated as the classic Hamiltonian problem
\begin{align}
  \begin{split}
    \label{eq:H-pq}
    \dot p & = - \frac{\partial H}{\partial q}(p,q), \\
    \dot q & =  \frac{\partial H}{\partial q}(p,q)
  \end{split}
\end{align}
for $p,q\in\R^n$, where $H$ is the \emph{Hamiltonian}
\begin{equation}
  \label{eq:Hamiltonian}
  H = \frac12p^2 + V(q).
\end{equation}

Mathematically, the existence of solutions to~\eqref{eq:H-pq}, often more succinctly written as
\begin{equation}
  \label{eq:H-z}
  \dot z = \begin{pmatrix} 0 & -\Id \\ \Id & 0 \end{pmatrix} H_z(z) \ , 
\end{equation}
with $z\defeq(p,q)$, is a classical problem. Periodic solutions have been a particular focus, and existence results obtained until the early 1980s are discussed in the beautiful survey article~\cite{Rabinowitz:82a}. Already for periodic solutions, a clear distinction has to be made for local results (that is, short time solutions) and global solutions describing solutions in the large. Apparently the first global global result was obtained by Seifert~\cite{Seifert:48a} for a Hamiltonian which is slightly more general than the one in~\eqref{eq:Hamiltonian}. The key idea of his proof is based on differential geometry, using an equivalent reformulation of~\eqref{eq:H-pq} in which solutions can be found as a geodesic in a (degenerate) Riemannian metric, the so-called Jacobi metric. A curve shortening procedure proposed by G.\ D.\ Birkhoff~\cite[Section V.7]{Birkhoff:66a} can then be applied to show the existence of a geodesic. This result has later been extended by Weinstein, and a more general result based on a different variational approach was given by Rabinowitz~\cite{Rabinowitz:78a}.

In the Sciences, the interest in non-periodic long time solutions has recently been rejuvenated by various applications. Namely, complex systems in physics, chemistry or biology can often be described by a potential energy landscape with many wells, separated by barriers. A common problem is then to find a trajectory joining a given initial point (configuration) with a given final point. We study this problem in the situation where the dynamics is determined by~\eqref{eq:hamildyn}, and the points given in~\eqref{eq:boundarycond} are potentially far apart. In particular, the two configurations will generically be located in different wells of the energy landscape. Rare events are an example of these transitions between two wells. Typically, thermally activated reactions have many deep wells separated by large energy barriers. Reactants will then spend most of the time jostling around in one well before a rare spontaneous fluctuation occurs that lifts the atoms of the reactant over the barrier into the next (product) valley. Information on rare events is crucial since they represent important changes in the system, such as chemical reactions or conformational modifications of molecules. A major challenge in Molecular Dynamics (MD) is that these hopping events take place so rarely that the computational limits of MD simulations can be easily exceeded. Since the problem~\eqref{eq:hamildyn}--\eqref{eq:boundarycond} is central in MD, a number of solution strategies have been proposed; see~\cite{Schwetlick:08b} for a brief review of some methods. Further, for practitioners of MD, the question arises whether any numerical approximation shadows a physical one~\cite{Gillilan:92a} (and if so, whether it shadows a generic physical trajectory). The lack of hyperbolicity rules out standard tools to prove shadowing (e.g.,~\cite[Theorem 18.1.3]{Katok:95a}). Thus, for MD, computations are ``based on trust''~\cite[Section 4.3]{Frenkel:02a}.

\begin{figure}
  \centering
  \includegraphics[angle=90,height=6cm]{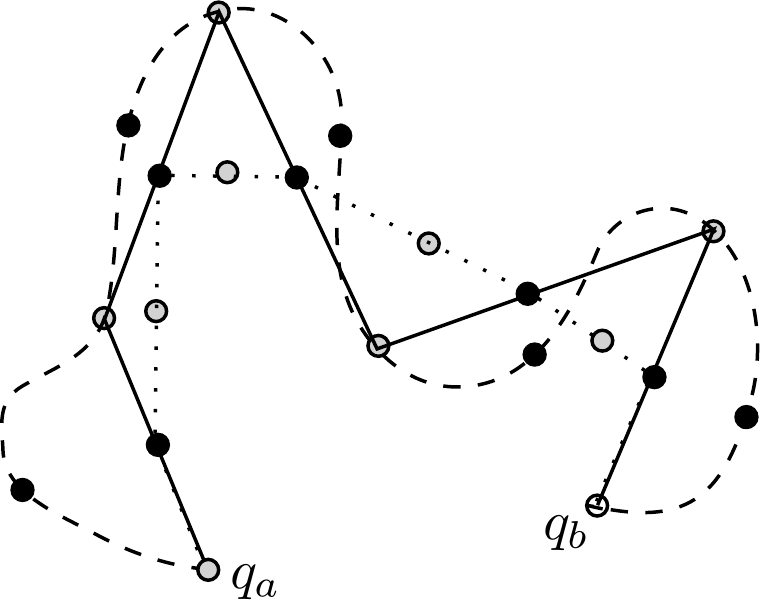}
  \caption{\label{fig:birkhoffeuclid} Birkhoff's algorithm, for the     toy example of the Euclidean metric in $\R^2$ and for $i=5$. The     initial curve is plotted as a dashed line. Points with odd index     are marked by black dots, points with even index by grey dots.  In     the first step, the points with even indices are kept fixed, and     joined by a geodesic. New positions for the points with odd     indices on the new curve are determined (solid curve). In a next     step, these points are joined by geodesics, which determines new     positions for the points with even indices (dotted curve). The     curves (slowly) converge to the geodesic line segment connecting     $q_0$ and $q_{2i}$.}
\end{figure}

One further difficulty for the computation of Hamiltonian trajectories, in particular for MD, is that these trajectories are often chaotic, and one has to restrict oneself to averaged statistical information. However, of particular interest in the analysis of rare events are trajectories going directly from one well of a potential to another one; such transition paths can be used to define so-called reaction coordinates. Often the most efficient algorithms using for example path sampling assume a knowledge of these coordinates and thus these non-chaotic trajectories are of significant practical importance. The method presented in this article is concerned with the calculation of such non-chaotic transition trajectories. Even these relatively ``simple'' trajectories are in practise very hard to compute, since they correspond to rare events and take place on very long time scales.

\subsection{Jacobi metric and Birkhoff curve shortening}
\label{sec:Birkh-curve-short}

Rather working with the Hamiltonian boundary value problem~\eqref{eq:hamildyn}--\eqref{eq:boundarycond} directly, we use an equivalent variational formulation, namely the Maupertuis principle, according to which trajectories to~\eqref{eq:hamildyn} with total energy $E$ are suitably re-parametrised geodesics with respect to the \emph{Jacobi metric}
\begin{equation}
  \label{eq:jacobi-metric-general}
  g_{ij}^\Jac(q) \defeq \(E - V(q)\) \delta_{ij} (q) 
\end{equation}
(more generally, if $Q$ is equipped with a Riemannian metric $g_{ij}$, then $g_{ij}^\Jac(q) \defeq \(E - V(q)\) g_{ij} (q)$). So Hamiltonian trajectories are critical points of the length functional associated with~\eqref{eq:jacobi-metric-general}. While the equivalence has been known for centuries, it seems that little advantage has been taken of the fact that this variational formulation has a very convenient mathematical structure. Note that Hamiltonian problems such as~\eqref{eq:hamildyn} are commonly indefinite, while a geodesic problem is elliptic and thus bounded from below.

So the problem reduces to that of finding geodesics in the Jacobi metric. This problem has been addressed by Birkhoff, who described a curve shortening procedure to find global geodesics under the assumption that local (sufficiently short) geodesics can be found explicitly. This assumption is not met here since local geodesics have to be approximated; the main task is to show that nevertheless global convergence can be obtained for a suitably devised local approximation scheme. Since the local scheme we propose also relies on a Birkhoff curve shortening idea, we present his idea in the global setting first to keep the presentation self-contained. Initially, one joins the given initial and final point $q_a$ and $q_b$ by an arbitrary curve. Then sufficiently many points are marked on the curve so that a local geodesic can be computed between next to nearest neighbours (see Figure~\ref{fig:birkhoffeuclid}). Note that the argument assumes that local geodesics can be computed with sufficient accuracy. In the first step, the points with even indices are kept fixed, and joined by a geodesic. New positions for the points with odd indices on the new curve are determined (solid curve in Figure~\ref{fig:birkhoffeuclid}). The procedure of joining next to nearest neighbouring points by local geodesics is then repeated for the points with odd indices (dotted curve in Figure~\ref{fig:birkhoffeuclid}). It is not hard to show that this iterative procedure decreases the length. Under suitable assumptions (e.g.,~\cite{Jost:02a}), this method can be shown to converge to a global geodesic; however, there are situations such as for degenerating metrics where convergence will not take place.

The central result of this article is an analogous local result, introducing a sequence of approximating sequences converging to a sufficiently short geodesic. The trade-off is that the result is local (applicable only to sufficient short geodesics), but does not assume that the approximating sequences consist of exact geodesic segments. We rely on the observation that the global Birkhoff argument localises the geodesic problem in a geometrically tractable way. That is, one can restrict the local analysis to points which are sufficiently close. We then show that they can be joined by a geodesic which in addition can be represented as a graph and prove consistency and convergence of the proposed local approximation.

\subsection{Results}
\label{sec:Results}

The main result of this article is a consistent approximation of what we call local geodesics; important aspects are that the proof is constructive and yields bounds on the allowed distance between points to be joined by local geodesics. The bounds are not in terms of the usual estimates from differential geometry (such as the injectivity radius, that is, the radius for which there is a unique geodesic starting at the centre, with arbitrary velocity), but are expressed explicitly in terms of the total and potential energy of the molecular Hamiltonian system~\eqref{eq:hamildyn}. We point out that locality of the geodesics does not necessarily require that the two end points are very close; an important aspect of the proof is that we may consider local geodesics which can be represented as graphs. Global geodesics can violate this assumption, while suitably small segments of geodesics remain geodesics and can be represented as graphs. Birkhoff's idea to segment an original connecting curve allows us to confine our algorithm to such local geodesics. The efficiency and applicability of a global Birkhoff method as in~\cite{Schwetlick:12a} will depend on the chosen parametrisation (number and location of points in Figure~\ref{fig:birkhoffeuclid}). For example, it is possible that refinements in a numerical implementation are required. However, the proof shows that in the setting studied in this article, no refinement or reparametrisation is required. 

As a by-product, we show the existence of a continuous (physical) trajectory for suitable points $q_a$, $q_b$, using the Jacobi formulation as Seifert~\cite{Seifert:48a}, but replacing the periodic boundary conditions considered there by Equation~\eqref{eq:boundarycond}. There are some related existence results~\cite{Gordon:74a,Benci:92a}, which also rely on the Jacobi formulation. The novelty of the results presented here are twofold: (i) while a formulation using the Jacobi metric is natural, a difficulty is that the metric degenerates at the boundary $\partial Q$ of the configuration space, where kinetic energy $\int\frac12 \dot q^2 \d x$ and potential energy $\int V(q) \d x$ agree. We provide \emph{a   priori} bounds to ensure that the geodesic stays away from $\partial Q$; the bounds depend on the location of the boundary points $q_a$ and $q_b$ or on the total energy $E$. Bounds could be obtained along the lines of thought presented in this paper, but in a simpler fashion. The reason why we give a more complicated argument is that the approximation we give here is constructive, giving existence of a solution and at the same time a consistent approximation procedure. Thus, we obtain an approximation procedure which may not necessarily be the most efficient but one for which can dispense with the need for trust.  Since the algorithm we develop is consistent, the issue of shadowing is answered in an affirmative way for the procedure we propose, under the assumptions made on the potential energy $V$ and the bounds in term of total energy $E$ made on the end points.

The existence of geodesics joining a given initial and final point in open domains $Q$, which is trivial within the radius of injectivity, is not obvious if the two points are further away from each other. We point out that the argument of this paper automatically proves the existence of such a geodesics for the Jacobi metric with sufficiently large total energy $E$. Other existence results, proceeding along quite different lines, can be found elsewhere~\cite{Gordon:74a,Benci:92a}. We remark that we will not address the question of how to choose $E$; this choice typically requires insight in the physics, chemistry or biology of the problem in question and thus cannot be answered in the general mathematical framework considered here. However, the existence result given here can be interpreted in two ways: given the initial point $q_a$ and the total energy $E$, the arguments provide estimates on possible locations of the final point $q_b$ so that $q_a$ and $q_b$ can be joined by a trajectory with total energy $E$. Alternatively, given $q_a$ and $q_b$, the analysis provides lower bounds on $E$ such that $q_a$ and $q_b$ can be joined with this total energy. It is easy to see that no general existence theorem can hold if $q_a$, $q_b$ and $E$ are unrestricted ($E$ determines the configuration manifold $Q$, and in particular for small $E$ the configuration manifold may be disconnected, so $q_a$ and $q_b$ could be in different components).

\subsection{Applications and limitations}
\label{sec:Appl-limit}

Approximations which are proven to be consistent, such as Godunov's scheme for hyperbolic equations, are often less efficient than algorithms for which consistency cannot be shown. The approximation introduced in this article is no exception. Yet, it is often possible to take inspiration from a consistent approximation and deduce efficient (but not provably consistent) formulations. This is the case for the formulation introduced here; a related flow model approximation has been shown to be able to detect different trajectories joining points in different wells of the energy landscape of the M\"uller potential and the collinear reaction $\mathrm{H}_2 + \mathrm{H} \to \mathrm{H} + \mathrm{H}_2$~\cite{Schwetlick:08b}. While our formulation relies both on the choice of Maupertuis' formulation and Birkhoff's curve shortening, which seems to be new to the field of numerical methods for Molecular Dynamics, the curve-shortening procedures resembles other rubber-band algorithms~\cite{Gillilan:92a}. For the isothermal case, a nudged elastic band method has been proposed by E, Ren and Vanden-Eijnden~\cite{E:02a}. There, the aim is to find minimal energy paths, which are defined as paths along which the orthogonal component of the deterministic vector field vanishes. The approach in~\cite{E:02a} to reduce the orthogonal contributions iteratively bears many similarities with the method presented here. As examples of string methods with temperature, we refer to the pioneering work by E, Ren and Vanden-Eijnden~\cite{E:03a,E:04a}.

Maupertuis' principle has been used before in MD~\cite{Banerjee:90a}, but without the connection to Birkhoff's curve shortening method. We also refer the reader to the recent work by Cameron, Kohn and Vanden-Eijnden which gives an analysis and in particular convergence results for a steepest descent string method~\cite{Cameron:11a}.

The usual numerical approach for a boundary value problem is a shooting method; there, the existence of a solution is assumed as well as the closeness of an initial guess to the solution. (There are abstract existence results available, see for example~\cite{Keller:92a}; however, as noted by Stoer and Bulirsch~\cite[7.3.3]{Stoer:90a}, the abstract formulation of the boundary conditions to be imposed rules out the condition~\eqref{eq:boundarycond}, for the first-order system considered there, even for the case $n=2$.) The method discussed in this paper provides both an explicit existence proof and estimates on the closeness required.

We point out that the Birkhoff algorithm also provides a strategy for gluing together local geodesics to obtain global paths, which converge to a global geodesic.  Note that this latter aspect of the Birkhoff approach defines a natural tool to localise the computation of geodesics which can be exploited for parallelisation. This aspect is discussed in more detail in~\cite{Schwetlick:12a}. An advantage of Birkhoff's   method is that these local steps are intrinsically parallelisable.

Finally, we mention connections to the Onsager-Machlup / Freidlin-Wentzell theory. There, an action functional is derived, with the minimal action path describing the most likely trajectory. The connection between that theory, Hamilton-Jacobi theory and the Maupertuis principle as discussed here is a topic with many open questions; we refer to~\cite{E:2004a} for results in this direction together with applications to rare events.

\subsection{Notation}
\label{sec:Notation}

Throughout the presentation, $Q$ is the \emph{configuration manifold} of a system and thus describes all possible states the system can occupy. The coordinates of the \emph{phase space} (cotangent bundle) $T^*Q$ are $\(q^j,p_j\)$, position and momentum. Analogously, the coordinates of the tangent bundle $TQ$ are $\(q^j, \dot{q}^j\)$, where $\dot{q}^j$ denotes the velocity. We assume that the system dynamics is conservative with $3N$ degrees of freedom. Then, the \emph{Hamiltonian} $H \colon T^*Q\to \R$ is defined as $H \defeq E \defeq T + V$. Here, the kinetic energy $T = T(p)$ is a function of the momenta only and $V=V(q)$ is the potential energy, depending on the coordinates $q$ alone. The \emph{Lagrangian} of the system is a function $L\colon TQ\to \R$, namely $L(q, \dot q) = T - V$. For a wide class of applications, it is sufficient to consider $\sp{\dot q}{\dot   q}=\sum_{j=1}^{3N} m_j \dot q_j \dot q_j$, the inner product for a system with $N$ particles with mass $m_j$, and to assume that the Lagrangian is of the standard form
\begin{equation}
  \label{eq:ParticularLagrangian}
  L\(q,\dot{q}\) \defeq \frac{1}{2} \sp{\dot{q}}{\dot{q}} - V(q).
\end{equation}

This paper is organised as follows: In Section~\ref{sec:continuous-setting}, we review the Maupertuis principle and the Birkhoff curve shortening algorithm, both for the continuous setting. Section~\ref{sec:Discr-sett-Birkh} describes the analogous discrete setting, contains the relevant \emph{a priori} estimates and introduces the discrete Birkhoff procedure for a fixed discretisation. Section~\ref{sec:Birkhoff-refinement} describes the Birkhoff refinement and convergence to the continuous limit. Numerical simulations and numerical convergence rates for a model problem are the content of Section~\ref{sec:Numer-investigations}.

\section{The continuous setting}
\label{sec:continuous-setting}

Our construction will be guided by a variational formulation, equivalent to~\eqref{eq:hamildyn}, where convergent approximations can be obtained with relative ease. This continuous setting is sketched in the present section.

It is a well known fact that solutions to~\eqref{eq:hamildyn} with pre-assigned total energy $E$ are re-parametrised geodesics in \emph{Jacobi's metric}~\eqref{eq:jacobi-metric-general}. This formulation is sometimes denoted \emph{Maupertuis' principle}, or \emph{Jacobi's least action principle}. For the special Lagrangian~\eqref{eq:ParticularLagrangian}, the \emph{Routhian} associated with Jacobi's principle is
\begin{equation}
  \label{eq:ParticularRouthian}
  R(q,q') = 2(E-V(q))\sp{q'}{q'} .
\end{equation}
Obviously, $R\(\gamma,\gamma'\)$ is a metric in those regions of $Q$ where $V(q) < E$. The action functional
\begin{equation}
  \label{eq:ActionRouthian}
  J[\gamma] \defeq \int_a^b R\(\gamma,\gamma'\) \d \tau
\end{equation}
is the measure of the length of $\gamma$ in this metric. For a given curve $\gamma$, the value $J[\gamma]$ is often called the \emph{energy} of $\gamma$; the \emph{length} of the curve is then
\begin{equation}
  \label{eq:LengthRouthian}
  L[\gamma] \defeq \int_a^b \sqrt{R\(\gamma,\gamma'\)} \d \tau .
\end{equation}
It is trivial to verify that critical points of the energy functional are critical points of the length functional; the converse is true for curves parameterised by arc length. The length functional is invariant under re-parametrisations, while a minimiser of the energy functional is automatically parameterised by arc length.

The Maupertuis' principle seeks geodesics, that is, stationary solutions of the functional~\eqref{eq:ActionRouthian} with the metric~\eqref{eq:ParticularRouthian}. Maupertuis' principle has been employed in a number of computational approaches and is regarded as a very accurate method for the verification of other algorithmic formulations~\cite{Olender:96a}.

We already mentioned in passing that solutions of~\eqref{eq:hamildyn} are re-parametrisations of solutions of~\eqref{eq:ActionRouthian} (or, equivalently,~\eqref{eq:LengthRouthian}). The re-parametrisation is such that the physical time for a solution of~\eqref{eq:hamildyn} can be recovered via the explicit formula
\begin{equation}
  \label{eq:time}
  t = \int_0^\tau \sqrt\frac{\sp{q'}{q'}}{2(E-V)}\d s.
\end{equation}

With the exception of~\cite{Schwetlick:08b}, numerical methods for~\eqref{eq:hamildyn} seem, to the best of our knowledge, not have taken advantage of the geodesic formulation~\eqref{eq:ParticularRouthian} (see~\cite{Schwetlick:08b} for a recent survey and a method that relies on observations similar to, but simpler than, those made in the present article). This is somewhat surprising, since the \emph{Birkhoff curve shortening   algorithm} is a classic method for the convergent approximation of geodesics.

\subsection{Existence of extended geodesics}
\label{sec:Exist-extend-geod}

Birkhoff's curve shortening method~\cite{Birkhoff:66a} is a constructive way to find extended geodesics, based on the assumption that local (short) geodesics (that is, geodesics joining points within the radius of injectivity) can be computed exactly. In this subsection, we recall the classic Birkhoff method; the main part of the paper then addresses the question of how to find the local geodesics constructively, for the case of the Jacobi metric. A straightforward implementation of the Birkhoff method relies on an approximation of local geodesics within the radius of injectivity and thus requires knowledge of the radius of injectivity.  For the complex energy landscapes we have in mind, this radius is not easily computable in a quantitative way.  Thus, we have here two aims. Firstly, we present an algorithm for the numerical approximation of local geodesics in an explicitly given neighbourhood. Secondly, we obtain a quantitative description of the size of this neighbourhood. A Birkhoff method then glues together these local geodesics to obtain an extended piecewise geodesic curve.

The aim of this article is to develop a discrete framework that mimics the Birkhoff procedure. Besides the usual difficulty of discretisation errors inherent in any numerical approach, we face the challenge that even the computation of local geodesics, as in Birkhoff's algorithm, is time-consuming and difficult to control for non-Euclidean metrics. We propose an approximation of Jacobi's metric~\eqref{eq:ParticularRouthian} by the trapezoidal rule. The key observation is that the difference between the Jacobi metric and the Euclidean metric occurs on a fine scale, described in greater detail below. The analysis of Section~\ref{sec:Discr-sett-Birkh} will show that crucial bounds on the discrete curvature in a Birkhoff procedure can be obtained since it is possible to show that in some (quantitatively characterised) situations the Birkhoff procedure for the Jacobi metric is locally identical to that for the Euclidean metric. The analysis will also reveal that for other configurations, that is, other geometric configurations, the two approximations differ, which results in the Jacobi procedure making steps which seem counter-intuitive if regarded within a Euclidean picture. Obviously, such a disagreement of the two approximations is necessary as we need to compute a geodesic in the Jacobi metric and thus have to differ at some point from the Euclidean picture.

\section{A  local discretised Birkhoff method}
\label{sec:Discr-sett-Birkh}

This section mimics the continuous framework laid out in Section~\ref{sec:continuous-setting} in a discrete setting.

\subsection{The discrete setting}
\label{sec:discrete-setting}

Throughout this section, we assume that the total energy $E$ is sufficiently large, as described in the next paragraph. We point out that the choice of $E$ determines the configuration manifold $Q\subset \R^n$, which we take as the set of points $q$ where $E-V(q)>0$. Let $q_a$ and $q_b\in Q$ be given; define $\ell = \frac{\abs{q_a-q_b}}2$, where $\abs{\cdot}$ is the Euclidean distance on $Q$ (not the Jacobi distance). We choose an orthonormal basis for $Q$ such that $q_b - q_a = 2\ell e_1$.  For $q \in Q$, we write $q = (X,Y)\in\R\times\R^{n-1}$, and in particular $e_1=(1,0)$.

Let us write the Jacobi metric in the form
\begin{equation}
  \label{eq:metric}
g_{ij}^\Jac(q)
  =e^{2h}(q)\delta_{ij} . 
\end{equation}

We require $E$ to be sufficiently large so that the line segment joining $q_a$ and $q_b$ is contained in $Q$ as well. In fact, we will work in a framework where either $E$ is chosen large enough, depending on the given points $q_a$ and $q_b$, or, given $E$, we choose points $q_a$ and $q_b$ in $Q$ with sufficiently small distance $\ell$, such that $q_b\in B_\ell(q_a)\subset Q$.

We consider a convex set $Q_H$ such that $Q_H$ is compactly contained in $Q$, $Q_H \subset\!\!\!\subset Q$; then the Jacobi metric is not degenerate, and hence Riemannian, on $Q_H$. It is bounded from above and from below by Euclidean metrics, but we do not use this fact. Then, if $V$ is sufficiently smooth, the metric factor $h$ is $C^{1,\al}(Q_H)$ and there is a finite $H\in\R $ such that the estimates
\begin{align}
  \abs{h(X+x,Y+y)
    -\(h(X,Y)+x\ddx
    h(X,Y)+y\cdot\nabla_Nh(X,Y)\)} 
  &\le H\(\abs x+\abs y\)^{1+\al}, \label{eq:metric-a} \\
  \abs{\nabla_Nh(X+x,Y+y)-\nabla_Nh(X,Y)}
  &\le H\(\abs x+\abs y\)^\al  \label{eq:metric-b}
\end{align}
hold.

Applying Equation~\eqref{eq:LengthRouthian} to the length of a straight line segment $\seg{q_1}{q_2}$ between two points,
\begin{equation*}
  \gamma_{\seg{q_1}{q_2}}(t):=q_1+t\(q_2-q_1\) \text{ for } t\in(0,1),
\end{equation*}
we obtain
\begin{equation}\label{eq:lineseg}
  L[\gamma_{\seg{q_1}{q_2}}] =
  \int_0^1 e^h (\gamma_{\seg{q_1}{q_2}}(t))  
  \norm{\dot\gamma_{\seg{q_1}{q_2}}(t)} \d t 
  =
  \int_0^1 e^h (\gamma_{\seg{q_1}{q_2}}(t)) \d t\ \norm{q_1-q_2}.
\end{equation} 

We now introduce the discretised setting. We first define an equidistant Cartesian grid on $Q_H \subset\R^n$.
We discretise the integral for the length of a straight line segment as follows.
\begin{definition}[Discretised length of segment]
  \label{def:dis-length}
  For the straight line segment $\seg{q_1}{q_2}$ between two points on   the grid, we define the \emph{discretised length} by applying the   2-point trapezoidal rule to~\eqref{eq:lineseg}
  \begin{equation*}
    \Lbar_{\seg{q_1}{q_2}}\defeq
\frac{e^{h} (\gamma_{\seg{q_1}{q_2}}(0))+e^{h} (\gamma_{\seg{q_1}{q_2}}(1))}2\norm{q_1-q_2}
=\frac{e^{h} (q_1)+e^{h} (q_2)}2\norm{q_1-q_2}.
  \end{equation*}
\end{definition}

Below we will introduce a discrete Birkhoff procedure which chooses to move points of polygonal curves in a direction normal to $e_1$ in order to decrease the length. It is thus necessary to estimate changes in the length as normal variations of a curve are considered. This is the content of the following lemma.

\begin{lemma}
  \label{lem:diff-quot-disc}
  Let $Q_H$ be convex set such that $Q_H \subset\!\!\!\subset Q$ and   assume the metric factor satisfies the H\"older   estimates~\eqref{eq:metric-a} and~\eqref{eq:metric-b} for $H>0$.   Let $q=(X,Y)\in Q_H$ be a grid point.

  Assume further that $\e\ne0$ and $\de,\De\in\R^{n-1}$ are such that   $\de\neq0$ and $q_\e=q+(\e,\De)$ and $q_\de=q+(0,\de)$ are also grid   points in $Q_H$ satisfying the bounds
  \begin{equation}
    \label{eq:basicass-disc}
    \abs{\frac\De\e}\le 1
    \text{\quad and\quad}
    \abs{\frac\de\e}\le 1.
  \end{equation}
  Then the difference quotient
  \begin{equation}
    \label{eq:DD-disc}
    \DDbar\e\de\De\defeq \frac{\Lbar_{\seg{q_\de}{q_\e}}-\Lbar_{\seg{q}{q_\e}}}{\abs{\e\de}}
  \end{equation}
  satisfies
  \begin{align*}
    \DDbar\e\de\De&=\frac\de{\abs\de}\cdot\(
    -\frac{\E{\frac\e2}\G\(\frac{\De}\e\)}\e
    +
    \frac12\nabla_Ne^h(q)\PG\(\frac{\De}\e\)
    \)
    +O\(\abs{\frac\de{\e^2}}\)+O\bigl(\abs\e^\al\bigr)
    ,
  \end{align*}
  where
  \begin{equation}
    \label{eq:G}
    q_{\frac\e2} = \frac{q+q_\e}2=\sse \text{ and }
   \PG(\eta)\defeq\sqrt{1+\abs{\eta}^2}
   ,\quad
    \G(\eta)\defeq\nabla\PG(\eta)=\frac\eta{\PG(\eta)}
    .
  \end{equation}
\end{lemma}
Note that $\abs{\G(\eta)}\le1$ for all $\eta\in\R$.

\begin{proof}
We can write
\begin{equation*}
  \frac1{\abs\e}\Lbar_{\seg{q_\de}{q_\e}}
  = 
  \frac{\E{\de}+\E{\e}}2
  \cdot\PG\(\frac{\De-\de}\e\).
\end{equation*}
Then it follows directly that
\begin{align*}
  &\DD\e\de\De=D_1S_2+S_1D_2
\end{align*}
with
\begin{align*}
  D_1&=\frac1{\abs\de}\left[\frac{\E{\de}+\E{\e}}2-\frac{\E{}+\E{\e}}2\right]
  =\frac1{2\abs\de}\left[\E{\de}-\E{}\right] 
  ,\\
  S_1&=\frac12\left[\frac{\E{\de}+\E{\e}}2+\frac{\E{}+\E{\e}}2\right]
  =\frac12\left[\E{}+\E{\e}\right]+\frac{\abs\de}2\cdot D_1,\\
  D_2&=\frac1{\abs\de}
  \(\PG\(\frac{\De-\de}\e\)-\PG\(\frac{\De}\e\)\),\\
  S_2&=\frac12 \(\PG\(\frac{\De-\de}\e\)+\PG\(\frac{\De}\e\)\)
  =\PG\(\frac{\De}\e\)+\frac{\abs\de}2D_2 .
\end{align*}
Let us define $D_{10}$ by writing
\begin{equation*}
  D_1=\frac12\frac\de{\abs\de}\cdot\nabla_N\E{}+D_{10},
\end{equation*}
where we estimate the error term $D_{10}$, using the H\"older bounds for the metric factor $h$
\begin{align*}
  \abs{D_{10}}&=\frac1{2\abs\de}\abs{\E{\de}-\E{}-\de\cdot\nabla_N\E{} }
  \le \abs{\de}^\al \frac H2.\\
\end{align*}
Secondly, we consider $S_1$ and write for the first term
\begin{equation*}
  \frac12\left[\E{}+\E{\e}\right]=\E{\frac\e2}+S_{10}
\end{equation*}
with $q_{\frac\e2}$  defined in~\eqref{eq:G} being the mid-point on the segment
$\seg{q}{q_\e}$. By symmetry, the H\"older bounds imply
\begin{align*}
  \abs{S_{10}}&
  \le H\(\abs\e+\abs\De\)^{1+\al}=
  H\(1+\abs{\frac\De\e}^{1+\al}\)\abs{\e }^{1+\al}\le 2H\abs{\e }^{1+\al}.
\end{align*}

Furthermore, we deduce
\begin{align*}
  D_2&=\frac1{\abs\de}
  \(\PG\(\frac{\De-\de}\e\)-\PG\(\frac{\De}\e\)\)
    =\frac1{\abs\de}\int^1_0F\(\frac{\De-\de            
      s}\e\)\cdot\frac{-\de}\e\d s
    ,
\end{align*}
where $\G$ is defined in~\eqref{eq:G}. Thus, we can write
\begin{align*}
  D_2&=-\frac\de{\abs\de}\cdot\frac1{\e} \G+ D_{20},
\end{align*}
where, using  $\abs{DF(\eta)}\le C$ for all $\eta\in\R$,
\begin{align*}
  \abs{D_{20}}&=\frac1{\abs\de}\abs{\int_0^1\int_0^s
\(\frac{-\de}{\e}\)^t\cdot
DF\(\frac{\De-\de \tau}\e\)
\cdot\(\frac{-\de}{\e}\)
\d\tau\d s}
\le
\abs{\frac\de{\e^2}}C.
\end{align*}

We summarise
\begin{align*}
  D_1S_2+S_1D_2
  &=
  \(\frac12\frac\de{\abs\de}\cdot\nabla_N\E{}+D_{10}\)
  \(\PG\(\frac{\De}\e\)+\frac{\abs\de}2D_2\)
  \\&
  +\(\E{\frac\e2}+S_{10}+\frac{\abs\de}2D_1\)
  \(-\frac\de{\abs\de}\cdot\frac1{\e} \G\(\frac{\De}\e\)+ D_{20}\)
  \\
  &=\frac\de{\abs\de}\cdot\(\frac12\nabla_N\E{}
  \PG\(\frac{\De}\e\)
  -\E{\frac\e2}\frac1{\e} \G\(\frac{\De}\e\)\)+C\abs{\frac\de{\e^2}}+O(\abs\e^\al)
  ,
\end{align*}
with the functions $\G$ and $\PG$ from~\eqref{eq:G}.
\end{proof}

We now prepare a crucial quantitative upper bound for the discrete bend of a polygon (Proposition~\ref{prop:rhombus}).  To this aim, consider three neighbouring points along the polygon. We assume that for this triplet $(q_{-\e},q,q_{+\e})$, the $X$ co-ordinates are distributed equidistantly for an $\e>0$,
\begin{equation*}
  q=(X,Y), 
\quad
  q_{\pm\e}=q+(\pm\e,\De_\pm) . 
\end{equation*}
For a given $\de\ne0$, we want to estimate the \emph{centred differences}
\begin{equation*}
  \Cbar(\e,\de)\defeq
  \(\Lbar_{\seg{q_{-\e}}{q+(0,\de)}}
  +\Lbar_{\seg{q+(0,\de)}{q_{+\e}}}   \)
  -\(\Lbar_{\seg{q_{-\e}}{q}} +
  \Lbar_{\seg{q}{q_{+\e}}}\).
\end{equation*}
Using centred coordinates, we can rewrite this as
\begin{equation*}
  \Cbar(\e,\de)=\e\abs{\de} \(\Dbar^{-}+\Dbar^{+}\),
\end{equation*}
where
\begin{equation*}
  \Dbar^{\pm}\defeq \DDbar{\pm\e}{\de}{\De_\pm}
\end{equation*}
with $\DDbar{\pm\e}{\de}{\De_\pm}$ defined in~\eqref{eq:DD-disc}. 

We will now combine the length calculation of Lemma~\ref{lem:diff-quot-disc} on both sides of the centre point $q$. We define, in analogy to~\eqref{eq:G},
\begin{align}
  q_{\frac{\pm\e}{2}}&=q+\tfrac 1 2 (\pm\e,\De_\pm) 
\label{eq:E-pm} \\
  \G_\pm& \defeq 
  \G\(\frac{\De_\pm}{\pm\e}\)=
  \frac{\frac{\De_\pm}{\pm\e}}{\sqrt{1+\abs{\frac{\De_\pm}{\pm\e}}^2}}.
  \label{eq:G-pm}
\end{align}
and obtain 
\begin{align}
\label{eq:Ceed}
\begin{split}
  \frac{\Cbar(\e,\de)}{\e\abs{\de}}
  &=\Dbar^{-}+\Dbar^{+}=\\
  &=\frac\de{\abs\de}\cdot\(
  -\frac{\E{\frac{-\e}2}\G_-}{-\e}
  +
  \frac12\nabla_Ne^h(q)\sqrt{1+\abs{\frac{\De_-}{-\e}}^2}
  \right.\\&\qquad\qquad\left.
  -\frac{\E{\frac\e2}\G_+}\e
  +
  \frac12\nabla_Ne^h(q)\sqrt{1+\abs{\frac{\De_+}{\e}}^2}\)
  \\&\quad
  +O\(\abs{\frac\de{\e^2}}\)+O\bigl(\abs\e^\al\bigr)
  \\
  &=\frac\de{\abs\de}\cdot\(
  - \frac{\E{\frac\e2}\G_+-\E{\frac{-\e}2}\G_-}{\e}
  +
  \nabla_Ne^h(q)\sqrt{1+\abs{\frac{\De_+}{\e}}^2}\)
  \\&\quad
  -
  \frac\de{\abs\de}\cdot\nabla_Ne^h(q)
  \(\frac{\sqrt{1+\abs{\frac{\De_+}{\e}}^2}
    -\sqrt{1+\abs{\frac{\De_-}{-\e}}^2}}2\)
  +O\(\abs{\frac\de{\e^2}}\)+O\bigl(\abs\e^\al\bigr)
  \\
  &=\frac\de{\abs\de}\cdot\(
  -\frac{\E{\frac\e2}\G_+-\E{\frac{-\e}2}\G_-}{\e}
  +
  \nabla_Ne^h(q)\sqrt{1+\abs{\frac{\De_+}{\e}}^2}\)
  \\&\quad
  -
  \frac\de{\abs\de}\cdot\nabla_Ne^h(q)
  \frac{\({\frac{\De_+}{\e}}+{\frac{\De_-}{-\e}}\)
    \({\frac{\De_+}{\e}}-{\frac{\De_-}{-\e}}\)}
  {2\(\sqrt{1+\abs{\frac{\De_+}{\e}}^2}
    +\sqrt{1+\abs{\frac{\De_-}{-\e}}^2}\)}
  \\&\quad
  +O\(\abs{\frac\de{\e^2}}\)+O\bigl(\abs\e^\al\bigr)
  .
  \end{split}
\end{align}
Let us remark that the term 
\begin{equation*}
  \frac1\e\({\E{\frac\e2}\G_+-\E{\frac{-\e}2}\G_-}\)
=
\frac1\e\({\E{\frac\e2}\G\(\frac{\De_+}{\e}\)-\E{\frac{-\e}2}\G\(\frac{\De_-}{-\e}\)}\)
\end{equation*}
is a proper     difference quotient for the discretisation length $\e$.

\begin{figure}
  \centering
  \includegraphics*[scale=0.5]{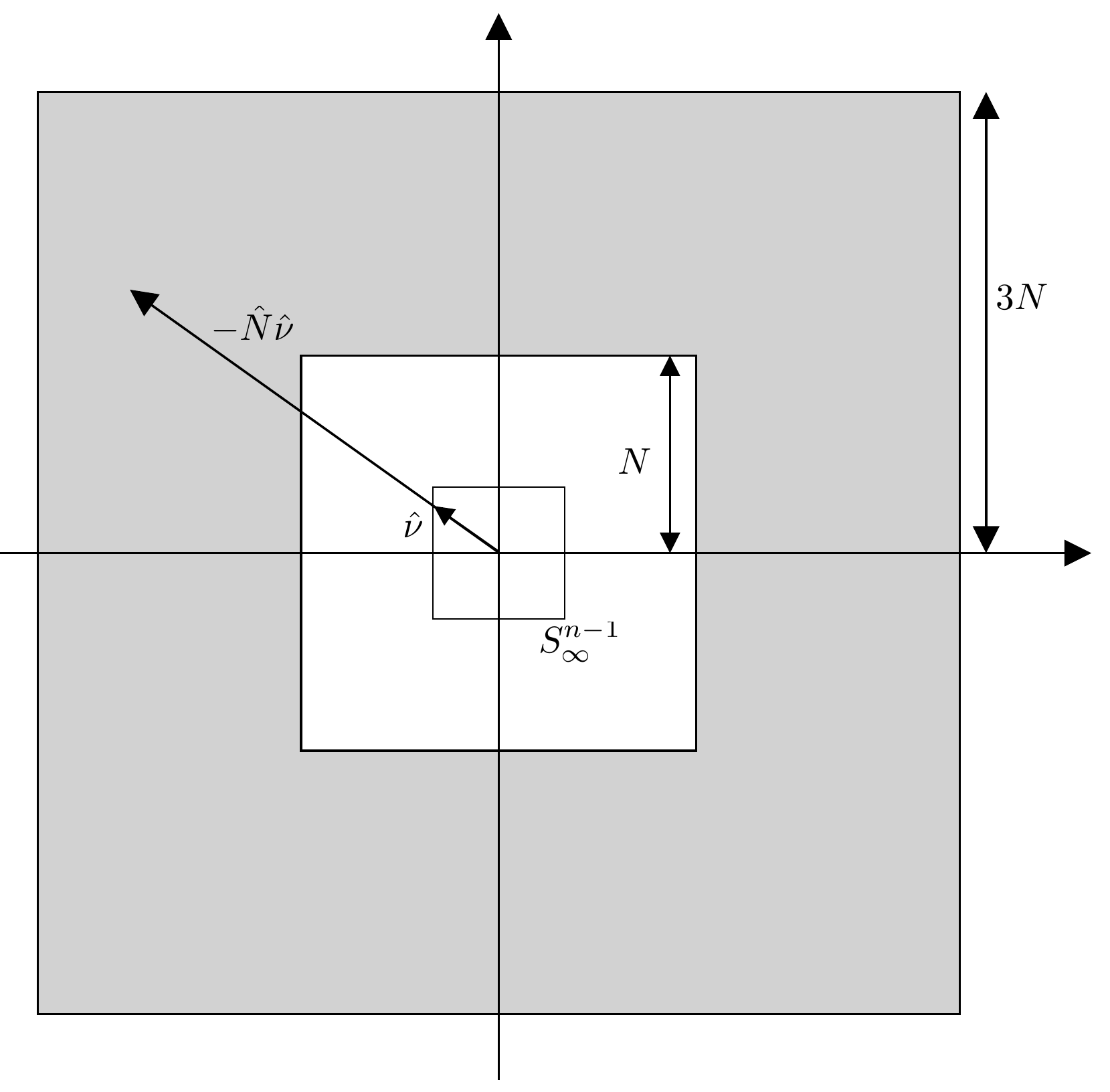}
  \caption{Geometric configuration as set out in     Proposition~\ref{prop:rhombus}}
  \label{fig:rhombus}
\end{figure}

The next result shows that the length of the edges of a three point polygon can be reduced by moving 
the middle point towards the line connecting the two outer points.
\begin{proposition}
  \label{prop:rhombus}
  Let $S_\infty^{n-1}$ be the $l^\infty$-sphere in $\R^{n-1}$,
  \begin{equation*}
    S_\infty^{n-1} \defeq \left\{\nnu \in \R^{n-1}
      \bigm| \sup_{j\in\{2, \dots, n\}} \abs{e_j \cdot \nnu} \le 1
      \text{ and } \abs{e_k \cdot \nnu} = 1  \text{ for some } 
      k \in       \{2, \dots, n\} \right\} .
  \end{equation*}
  There exists $N>1$ and $\e_0=\e_0(N)\in(0,1/N)$ such that for all $\e   \in(0,\e_0)$ and all triplets
  \begin{equation*}
  \(q_{-\e},q,q_\e\)
\quad\text{with}\quad
  q_{\pm\e}=q+(\pm\e,\De_\pm)
  \end{equation*}
  which satisfy
  \begin{align}
    \abs{\frac{\De_\pm}{\pm\e}}&\le 1 \label{eq:F-bound}
    \intertext{and}
    \frac{\De_++\De_-}{\e^2}&= -\hat N\hat\nnu, 
    \text{ with $\hat N\in\(N, 3N\)$ and $\hat\nnu\in S_\infty^{n-1}$,  
      see Fig.~\ref{fig:rhombus}},
    \label{eq:av-bound}
  \end{align} 
  there holds
  \begin{equation*}
    \begin{cases}
      \Cbar(\e,\de)>0 & \text{for every } \de\text{ with }      
      \abs\de\le\e^{\al+2} \text{ and }\frac{\de}{\abs{\de}}
      \cdot\hat\nnu=1, \\
      \Cbar(\e,\de)<0 & \text{for every } \de\text{ with }
      \abs\de\le\e^{\al+2} \text{ and }\frac{\de}{\abs{\de}}
      \cdot\hat\nnu=-1.
    \end{cases}
  \end{equation*}
\end{proposition}
The important implication of this statement is that length-reducing procedures for triplets will not increase the discrete curvature indefinitely. Specifically, if the curvature of a triplet is such that it is in the white inner square in Fig.~\ref{fig:rhombus}, then the length shortening procedure may increase the curvature (unlike in the Euclidean case). However, if the curvature increases, it will eventually enter the grey region in Fig.~\ref{fig:rhombus}. Then the proposition shows that a further step ``outwards'' (increasing the discrete curvature) necessarily increases the length, while the corresponding ``inward'' step decreases the length. This prevents the discrete curvature to grow without bounds under a length shortening process. So in the inner white region, the Birkhoff procedure for the Jacobi metric and for the Euclidean metric can differ; in fact they have to differ at some point since the results, the respective geodesics, differ.

\begin{proof}
From~\eqref{eq:Ceed} we deduce, as $\abs{\frac{\de}{\e^2}}\le\abs{\e}^\al$, 
\begin{align}
  \begin{split}
  \frac{\Cbar(\e,\de)}{\e\abs{\de}}
  &=\frac\de{\abs\de}\cdot\(
  -\frac{\E{\frac\e2}\G_+-\E{\frac{-\e}2}\G_-}{\e}
  +
  \nabla_Ne^h(q)\sqrt{1+\abs{\frac{\De_+}{\e}}^2}\) 
  \\&\quad
  -
  \frac\de{\abs\de}\cdot\nabla_Ne^h(q)
  \frac{\({\frac{\De_+}{\e}}+{\frac{\De_-}{-\e}}\)
    \({\frac{\De_+}{\e}}-{\frac{\De_-}{-\e}}\)}
  {2\(\sqrt{1+\abs{\frac{\De_+}{\e}}^2}
    +\sqrt{1+\abs{\frac{\De_-}{-\e}}^2}\)}
  +O\bigl(\abs\e^\al\bigr) . 
  \end{split}
  \label{eq:ceed-est-a}
\end{align}
We rewrite
\begin{align}
  \label{eq:hamster}
  \E{\frac\e2}\G_+-\E{\frac{-\e}2}\G_- &= 
\frac{ \E{\frac\e2}+  \E{\frac{-\e}2}}2\(\G_+-\G_-\) +
  \frac{\G_++\G_-}2\(\E{\frac\e2}-  \E{\frac{-\e}2}\). 
\end{align}
Firstly, with the identities
\begin{equation}
  \label{eq:diff-av}
  \frac{\De_\pm}{\pm\e}=\frac{\De_+-\De_-}{2\e}\pm\frac{\De_++\De_-}{2\e}
  ,
\end{equation}
we infer for the difference of $\G_\pm$
\begin{align*}
&\(\G_+-\G_-\)=
\frac{\frac{\De_+}{\e}}{\sqrt{1+\abs{\frac{\De_+}{\e}}^2}}
-
\frac{\frac{\De_-}{-\e}}{\sqrt{1+\abs{\frac{\De_-}{-\e}}^2}}
=\frac{\frac{\De_+-\De_-}{2\e}+\frac{\De_++\De_-}{2\e}}{\sqrt{1+\abs{\frac{\De_+}{\e}}^2}}
-
\frac{\frac{\De_+-\De_-}{2\e}-\frac{\De_++\De_-}{2\e}}{\sqrt{1+\abs{\frac{\De_-}{-\e}}^2}}\\
&=\frac{\De_++\De_-}{2\e}\(\frac1{\sqrt{1+\abs{\frac{\De_+}{\e}}^2}}+\frac1{\sqrt{1+\abs{\frac{\De_-}{-\e}}^2}}\)
+
\frac{\De_+-\De_-}{2\e}\(\frac1{\sqrt{1+\abs{\frac{\De_+}{\e}}^2}}-\frac1{\sqrt{1+\abs{\frac{\De_-}{-\e}}^2}}\)
\\
&=\frac{
\frac{\De_++\De_-}{\e}\frac{\sqrt{1+\abs{\frac{\De_+}{\e}}^2}+\sqrt{1+\abs{\frac{\De_-}{-\e}}^2}}2
-\frac{\De_+-\De_-}{2\e}\frac{\abs{\frac{\De_+}{\e}}^2-\abs{\frac{\De_-}{-\e}}^2}
{\(\sqrt{1+\abs{\frac{\De_+}{\e}}^2}+\sqrt{1+\abs{\frac{\De_-}{-\e}}^2}\)}
}
{\sqrt{1+\abs{\frac{\De_+}{\e}}^2}\sqrt{1+\abs{\frac{\De_-}{-\e}}^2}}
=A
\frac{\De_++\De_-}{\e},
\end{align*}
where
\begin{equation}
  \label{eq:A}
  A=\frac{
    \(\frac{\sqrt{1+\abs{\frac{\De_+}{\e}}^2}+
      \sqrt{1+\abs{\frac{\De_-}{-\e}}^2}}2\)^2\Id
    -\(\frac{\De_+-\De_-}{2\e}\)\(\frac{\De_+-\De_-}{2\e}\)^T
  }
  {\sqrt{1+\abs{\frac{\De_+}{\e}}^2}\sqrt{1+\abs{\frac{\De_-}{-\e}}^2}
    \frac{\sqrt{1+\abs{\frac{\De_+}{\e}}^2}
      +\sqrt{1+\abs{\frac{\De_-}{-\e}}^2}}2
  }
  .
\end{equation}
Secondly,~\eqref{eq:metric-a},~\eqref{eq:metric-b} and~\eqref{eq:F-bound} imply
\begin{align*}
  \frac{\E{\frac\e2}-  \E{\frac{-\e}2}}{e^h(q)} 
  & \le \e^{1+\al} H \cdot 2^{1+\al} + \e\norm{\nabla h}\cdot 2
.
\end{align*}
These two steps imply for~\eqref{eq:hamster}
\begin{align*}
  \abs{\E{\frac\e2}\G_+-\E{\frac{-\e}2}\G_-
 - \frac{ \E{\frac\e2}+  \E{\frac{-\e}2}}2
 A \frac{\De_++\De_-}{\e} }&\le
  e^h(q)\( 2^{1+\al} \e^{1+\al} H  + 2 \e\norm{\nabla     h}\) \frac 1{\sqrt{2}} , 
\end{align*}
where we used the bound $\abs{\G^\pm}\le\frac 1{\sqrt{2}}$ implied by~\eqref{eq:F-bound}.

We return to~\eqref{eq:ceed-est-a}, and with $A$ from~\eqref{eq:A}, together with~\eqref{eq:av-bound} and $\abs{\frac{\De_+-\De_-}{\e}}=\abs{\frac{\De_+}{\e} +   \frac{\De_-}{-\e}}\le 2$, we obtain
\begin{alignat*}2
\frac{\Cbar(\e,\de)}{\e\abs{\de}} 
& \ge 
\frac{ \E{\frac\e2}+  \E{\frac{-\e}2}}2
&&\left[ -\frac\de{\abs{\de}} A \frac{\De_++\De_-}{\e^2}
- \(1+\norm{\nabla h}\cdot 2\frac 1{\sqrt{2}}\)
  -\norm{\nabla h}\sqrt{2} 
\right. \\&&&\ \left.
-\norm{\nabla h}
  \frac 1{2\sqrt{2}} \abs{\frac{\De_++\De_-}{\e}}
\right]
  +O(\abs{\e}^\al) 
\\
& \ge 
\frac{ \E{\frac\e2}+  \E{\frac{-\e}2}}2
 &&\left[ \frac\de{\abs{\de}} A\hat\nnu \cdot\hat N
- \(1+\norm{\nabla h}\cdot \frac 2{\sqrt{2}}\)
  -\norm{\nabla h}\sqrt{2} 
\right. \\&&&\ \left.
-\e \hat N\norm{\nabla h}
  \frac 1{\sqrt{2}}\right]
  +O(\abs{\e}^\al) 
.
\end{alignat*}
We use the inequality
\begin{equation*}
  \sqrt{1+\abs{a}^2}\sqrt{1+\abs{b}^2}\ge1+a_ib_i 
\end{equation*}
for all $a,b\in\R^{n-1}$ to deduce for a diagonal element of the matrix $A$
\begin{align*}
  e_i^TAe_i\ge\frac{
    \frac12\(1+\sqrt{1+\abs{\frac{\De_+}{\e}}^2}
    \sqrt{1+\abs{\frac{\De_-}{-\e}}^2}
    +{\frac{e_i^T\De_+}{\e}}\cdot{\frac{e_i^T\De_-}{\e}}\)
  }
  {\sqrt{2}^3}
  \ge\frac1{\sqrt{2}^3}.
\end{align*}
Thus, we obtain in the case $\frac\de{\abs{\de}} \cdot\hat\nnu =1$
\begin{equation*}
  \frac{\Cbar(\e,\de)}{\e\abs{\de}} \ge
  \frac{\frac{ \E{\frac\e2}+  \E{\frac{-\e}2}}2}{\sqrt{2}^3} \left\{\hat N-
    \left[\sqrt{2}^3+\norm{\nabla h} \cdot 2
      \(
      2 + 2
      +\e \hat N
      \)\right]\right\}
  +O(\abs{\e}^\al)
  .
\end{equation*}
Hence
\begin{align*}
  \frac{\Cbar(\e,\de)}{\e\abs{\de}}
  \ge
  \frac{
\frac{ \E{\frac\e2}+  \E{\frac{-\e}2}}2}2
  \(\hat N-2\(4+\norm{\nabla h}\Big[4+\hat N\e\Big]\)\)
  +O(\abs{\e}^\al)
  .
\end{align*}
Thus, we may take $N=3\(4+5\norm{\nabla h}\)\ge12$ and choose $\e_0$ smaller than $1/N$.  
As $\hat N\ge N$ and as we may reduce $\e_0$ possibly further to compensate the error term, 
we can ensure the positivity of $\frac{\Cbar(\e,\de)}{\e\abs{\de}}$ for all $0<\e<\e_0$.

Analogously, we deduce in the case $\frac\de{\abs{\de}} \cdot\hat\nnu =-1$ that there holds
\begin{align*}
  \frac{\Cbar(\e,\de)}{\e\abs{\de}}
  \le
  \frac{
\frac{ \E{\frac\e2}+  \E{\frac{-\e}2}}2}{2}
  \(-\hat N+2\(4+\norm{\nabla h}\Big[4+\hat N\e\Big]\)\)
  +O(\abs{\e}^\al)
  <0,
\end{align*}
that is, strict negativity for the opposite sign.
\end{proof}

\subsection{Birkhoff method for a fixed discretisation}
\label{sec:Birkh-meth-fixed}

We recall that, for given $q_a$ and $q_b$ and $\ell = \frac{\abs{q_a-q_b}}2$, where $\abs{\cdot}$ is the Euclidean distance on $Q$ (not the Jacobi distance), we have chosen an orthonormal basis for $Q \subset \R^n$ such that $q_b - q_a = 2\ell e_1$. For $q \in Q$, we write $q = (X,Y)\in\R\times\R^{n-1}$.

\begin{definition}[Polygon, associated points and differences]
  \label{def:poly}
  Let $M\in\N$ large enough such that $\e\defeq\frac   \ell{M}\le\hat\e$. We define a polygon with $2M+1$ vertices as
  \begin{align}
    \label{eq:pointgraph}
    q_j =\(X_j,Y_j\)
    \defeq \frac 12 \(q_b - q_a\) 
    +  \(X_j e_1 + \sum_{k=1}^{n-1} Y_k e_{k+1} \)
    \text{ with } X_j =   j \e
    \text{ for }  j = -M, \dots, M . 
  \end{align}
  Note that $q_a = q_{-M}=\(X_{-M},0\)=\(-\ell,0\)$ and   $q_b=q_M=\(X_{M},0\)=\(\ell,0\)$. The \emph{polygon} $\gamma$   associated with these points consists of the line segments joining   neighbouring points.
  
  We then define the set of interior nodes $\JJ \defeq \{j\in\Z \bigm|   \abs j < M\}$ and set for all $j\in\JJ$
  \begin{align*}
    \De^\pm_j=Y_{j\pm1}-Y_j.
  \end{align*}
\end{definition}

\begin{remark}
  Note that $\De^+_j=-\De^-_{j+1}$ and
  \begin{equation*}
    \De^-_j+\De^+_j=Y_{j-1}+Y_{j+1}-2Y_j.
  \end{equation*}
\end{remark} 

We now show that the fixed boundary points $q_{\pm M}$ and the estimate on the second differences from Proposition~\ref{prop:rhombus} ensure that the first differences remain bounded by $1$.
\begin{definition}
  \label{def:PtwoN}
  Consider a polygon $\gamma$ as in Definition~\ref{def:poly}. Let $N$   be given by Proposition~\ref{prop:rhombus}. We say that   $\gamma\in\PPtwo$ if the second difference quotients of $\gamma$   satisfy
  \begin{align}
    \label{eq:second-bound}
    \max_{j\in \JJ}\max_{i=2, \dots, n}     \abs{e_i            
      \cdot\frac{\De^{-}_j+\De^{+}_j}{\e^2}} \le 2 N.
  \end{align}
\end{definition}

\begin{lemma}
  \label{lem:par-nbhd}
  Let $\ell\leq \ell_0 \defeq\frac1{4N\sqrt{n-1}}$, where $N$ is given   by Proposition~\ref{prop:rhombus}.  Let $\ga$ be a polygon as in   Definition~\ref{def:poly}.  If $\gamma\in\PPtwo$ (see   Definition~\ref{def:PtwoN}), then there holds
  \begin{align}
     \sup_{j\in\JJ} \abs{\frac{\De^{\pm}_j}{\pm\e}}&\le 1 
    \label{eq:first-bound}
  \end{align}
  as well as $\gamma$ is contained in
  \begin{equation}
    \label{eq:S}
    \PN \defeq\left\{q=(X,Y): X\in\left[-\ell,\ell\right],
    \abs Y\le N\( \ell^2-\abs X^2\)\right\}.
  \end{equation}
\end{lemma}

\begin{proof}
We add up the second differences:
As  $Y_{\pm M} =0$, any  unit vector $\nnuu$, $k=2,\ldots,n$, orthogonal to $e_1$ 
satisfies
 $\nnuu\cdot Y_{\pm M}$ =0. Hence, for $\De^+_j=
Y_{j+1}-Y_j$ it follows that
\begin{equation*}
  0=\nnuu \cdot  
  \(Y_M-Y_{-M}\)=\sum_{j\in\JJ\cup\{-M\}}\nnuu\cdot\De^+_j.
\end{equation*}

By the Mean Value Theorem, there exists a $j_-\in\JJ\cup\{-M\}$ such that $\nnuu\cdot\De^+_{j_-}\le0$. For any $j\in\JJ\cup\{-M\}$ we thus deduce from the boundedness of the second difference quotients
\begin{align*}
  \nnuu\cdot\De^+_j
  &=\nnuu\cdot\De^+_{j_-}+\sign\(j-j_-\)
  \sum_{l=\min\{j,j_-\}}^{\max\{j,j_-\}-1}\nnuu\cdot\(\De^+_{l+1}-\De^+_l\)
  \\
  &=\nnuu\cdot\De^+_{j_-}+\sign\(j-j_-\)
  \sum_{l=\min\{j,j_-\}}^{\max\{j,j_-\}-1}\nnuu\cdot\(\De^+_{l+1}+\De^-_{l+1}\)
  \\&
  \le\abs{j-j_-}\cdot 2N\e^2
  \le 2M\cdot 2N\e^2= 2\ell\cdot 2N\e
  .
\end{align*}
We proceed analogously for a $j_+\in\JJ\cup\{-M\}$ with $\nnuu\cdot\De^+_{j_+}\ge0$ and obtain for all $j\in\JJ\cup\{-M\}$
\begin{equation*}
  \abs{\nnuu\cdot\frac{\De^+_j}{\e}}\le4N\ell.
\end{equation*}
As $\nnuu$ is arbitrary, we find $\abs{\frac{\De^+_j}{\e}}\le 4N\ell\sqrt{n-1}$. As $\ell\leq \ell_0$, we conclude
\begin{equation*}
  \abs{\frac{\De^+_j}{\e}}\le 1.
\end{equation*}

As $\De^+_j=-\De^-_{j+1}$, the estimate extends to the corresponding `negative' differences $\De^-_j$. Note that an integration of the second differences yields directly the inclusion $\gamma\in \PN$.
\end{proof}

We now define the Birkhoff method for fixed $\epsilon>0$. We start with a polygon $\gamma^0$ represented by the vertices $q_j^0$ that can be written as in~\eqref{eq:pointgraph}. We think of the Birkhoff method as an iterative process to update a polygon $\gamma^l$ to a polygon $\gamma^{l+1}$ which has a strictly smaller discrete length. It will be shown later that such an update is, for a fixed discretisation, not always possible and the Birkhoff method thus terminates.

\begin{definition}[Birkhoff step]
  \label{def:birkhoff}
  We consider a polygon $\gamma^l$ represented by $2M+1$ points   $q_j^l$, as in Definition~\ref{def:poly}. Let $0< \ze\le\e^{2+\al}$   and
  \begin{equation*}
    \NN\defeq\{\sigma e_i \bigm| 
    \sigma \in\{\pm 1\} \text{ and } i= 2, \ldots, n \} . 
  \end{equation*}
  Then
  \begin{enumerate}
  \item \label{i:1} we consider sequentially every $j\in\JJ$. For     given $j$, consider sequentially every $\nnu\in\NN$ and try to     move the interior point $q^l_j$ to $q_j^*=q^l_j + \de$, with     $\de\defeq\ze \nnu$, to achieve
    \begin{equation*}
      \Lbar_{\seg{q^l_{j-1}}{q_j^*}} +
      \Lbar_{\seg{q_j^*}{q^l_{j+1}}} < \Lbar_{\seg{q^l_{j-1}}{q_j^l}}+
      \Lbar_{\seg{q_j^l}{q^l_{j+1}}}.
    \end{equation*}
    That is, the passage via $q_j^*$ is shorter than via the original
    $q_j^l$. In the affirmative case, then we define the update
    $\gamma^{l+1}$ as
    \begin{equation*}
      q_j^{l+1}\defeq q_j^* \text { and }  
      q_k^{l+1}\defeq q_k^l \text { for } k \ne j.
    \end{equation*} 
    Thus the update has strictly smaller discrete length.
  \item If~\ref{i:1} is not affirmative for any $j\in\JJ$, then     Birkhoff step is called \emph{void}.
  \end{enumerate}
  The Birkhoff step depends on the sequential order chosen for $\JJ$   and $\NN$. Here and later, we regard this choice as fixed and hence   the Birkhoff step is uniquely defined.
\end{definition}

It is immediate that if $q_j^l$ is on the grid defined in Section~\ref{sec:discrete-setting}, then $q_j^{l+1}$ lies on the grid as well; the Birkhoff step thus makes only movements which are compatible with the grid, and thus results in polygons with vertices on the grid.

\begin{definition}[Birkhoff method and map]
  The \emph{Birkhoff method} is the iteration obtained by consecutive   Birkhoff steps starting with a polygon $\gamma^0$ until the Birkhoff   step is void.

  The \emph{Birkhoff map} maps a starting polygon $\gamma^0$ to the   final polygon obtained by the Birkhoff method.
\end{definition}

\begin{proposition}
  \label{prop:par0inv}
  Let $\ell_0>0$ be given by Lemma~\ref{lem:par-nbhd} with $N>1$ be   given by Proposition~\ref{prop:rhombus}.  Consider an initial   polygon $\gamma^0$ as in Definition~\ref{def:poly} with sufficiently   close endpoints, that is $\ell\le\ell_0$.  Furthermore, assume   $\gamma^0\in\PPtwo$, that is, the second differences are bounded.  The   updates of the Birkhoff steps obey the same assumptions, that is the   remain {\em graph-like} polygons in the sense of   Definition~\ref{def:poly} and are contained in $\PPtwo$.   Furthermore, the Birkhoff method terminates after finitely many   steps with a final polygon in the strictly smaller set $\PP$, that   is, there holds (recall Definition~\ref{def:PtwoN})
  \begin{align*}
    \max_{j\in \JJ}\max_{i=2, \dots, n}     
    \abs{e_i       \cdot\frac{\De^{-}_j+\De^{+}_j}{\e^2}} &\le N.
  \end{align*}
\end{proposition}
We point out and will use later that the second differences bound associated with $\PP$ is exactly half the bound associated with $\PPtwo$. This will be crucial to compensate for the doubling of the discrete curvature if one halves the stepsize.
\begin{proof}
Firstly, we notice that  $\gamma^0\in\PN$ by Lemma~\ref{lem:par-nbhd}.
Consider a Birkhoff step 
according to Definition~\ref{def:birkhoff}.
(i) If the step is {\em void}, then $\ga^0$ is the final polygon.
The final polygon has to lie in $\PP$ as otherwise Proposition~\ref{prop:rhombus} would ensure the existence of a further affirmative Birkhoff step, that is, an update site $q_j^l$ and an associated direction $\de$ such that $\Cbar(\e,\de)<0$.

(ii)
Otherwise, consider a single affirmative Birkhoff step acting on a polygon $\gamma^l\in\PPtwo$, that is, there is a $j\in\JJ$ and $\nnu\in\NN$ such that
\begin{equation*}
  \Lbar_{\seg{q^l_{j-1}}{q_j^l+\ze\nnu} }+
  \Lbar_{\seg{q_j^l+\ze\nnu}{q^l_{j+1}}} < 
  \Lbar_{\seg{q^l_{j-1}}{q_j^l} }+
  \Lbar_{\seg{q_j^l}{q^l_{j+1}}}.
\end{equation*}
For $q\defeq q_j^l$, $q_\pm\defeq q_{j\pm1}^l$ and $\de=\ze\nnu$, it follows that
\begin{equation}
  \label{eq:C-neg}
  \Cbar(\e,\de)=
  \(\Lbar_{\seg{q_{-}}{q+(0,\de)}}
  +\Lbar_{\seg{q+(0,\de)}{q_{+}}}   \)
  -\(\Lbar_{\seg{q_{-}}{q}} +
  \Lbar_{\seg{q}{q_{+}}}\)<0.
\end{equation}

We want to show that the Birkhoff step leaves $\PPtwo$ invariant. 

\emph{Case 1:} If 
\begin{equation*}
  \max_{j\in \JJ}\max_{i=2, \dots, n} \abs{e_i       \cdot\frac{\De^{-,l+1}_j+\De^{+,l+1}_j}{\e^2}} \le 2N,
\end{equation*}
then by definition of $\PPtwo$ the update $\ga^{l+1}$ is contained in $\PPtwo$, so nothing is to be shown.

\emph{Case 2:} Now we assume on the contrary that $\ga^{l+1}$ is not in the set $\PPtwo$. That is, there exist a $j\in \JJ$ and an $i\in\{2, \ldots, n\}$ such that
\begin{equation*}
  \abs{e_i \cdot \frac{\De^{-,l+1}_j+\De^{+,l+1}_j}{\e^2}}>2N;
\end{equation*}
thus we can write
\begin{align*}
  \frac{\De^{-,l+1}_j+\De^{+,l+1}_j}{\e^2}
  &=\frac{\De^{-,l}_j+\De^{+,l}_j}{\e^2}-2\frac\de{\e^2}
  \defqe-\hat N \hat \nnu -2 \frac \ze {\e^2} \nnu.
\end{align*}
In view of the notation used in Proposition~\ref{prop:rhombus}, we write
\begin{align*}&
  -\hat N \hat \nnu =\frac{\De^{-}+\De^{+}}{\e^2},
\end{align*}
where $\De_\pm\defeq\De_j^{\pm,l}$ with $\abs{\frac{\De_\pm}{\pm\e}}\le1$ by Lemma~\ref{lem:par-nbhd}. Furthermore, there holds $\hat N\in(N, 2N]$, $\hat\nnu\in S^{n-1}_\infty$ and $\de=\ze\nnu$, $\nnu\in\NN$. Note that $\hat\nnu \cdot \nnu$ is restricted to the values $-1,0,1$ by the definitions of $S^{n-1}_\infty$ and $\NN$.  By assumption, however, $\ga^{l+1}$ has left the set $\PPtwo$, hence $\hat\nnu \cdot \nnu=1$. Proposition~\ref{prop:rhombus} then implies $\Cbar(\e,\de)>0$. This contradicts~\eqref{eq:C-neg}, so Case 2 is in fact impossible.

Using the inclusion $\PPtwo\subset S$, we conclude that for a finite discretisation length $\e$, the number of distinguished polygons represented on the discrete grid is finite. Each affirmative Birkhoff step is strictly reducing the length, thus does not allow us to visit the same polygon twice. Hence the method terminates after finitely many affirmative Birkhoff steps.  As already argued in (i), the final polygon has to lie in $\PP$ as otherwise one could show the existence of a further affirmative Birkhoff step.
\end{proof}

\section{Birkhoff refinement}
\label{sec:Birkhoff-refinement}

This section consists of three parts.  In the first part, we define a sequence of polygons $\gamma_{k}$, which are the the final polygons of the Birkhoff method for the discretisation length $\e_k$. As $\e_k\to0$, we will show that the curves $\gamma_{k}$ converge to a curve $\gamma$. In the second part, we recall a weak formulation of the geodesic equation.  In the third part, we show that the limit $\gamma$ satisfies this weak geodesic equation, hence $\ga$ is smooth and a stationary curve for the Jacobi length functional.

\subsection{Refinement and convergence}
\label{sec:Refin-conv}

\begin{definition}[Refinement]
  \label{def:refine}
  Let $M_0$ be large enough such that $\e_0\defeq\frac   \ell{M_0}\le\hat\e$ and $\ze_0\defeq\e_0^3$.  Let us define the   starting polygon $\gamma_0$ as the $\e_0$--discretisation of the   straight line segment with endpoints $q_a$, $q_b$.

For $k=1,2, \ldots$, we want to halve the discretisation length, that is, $\e_k\defeq\frac \ell{M_k}$ for $M_k\defeq 2M_{k-1}$.  Let $\JJ_k\defeq \{j\in\Z \bigm| \abs j < M_k\}$.  We embed the polygon $\gamma_{k-1}$ into the finer grid by introducing new vertices at the midpoints of the connecting line segments. In the notation of Definition~\ref{def:poly}, this means the embedded polygon $\bar \gamma_k$ has the vertices
  \begin{alignat*}2
    \bar q_{j}^{k} &= q_{\frac j2}^{k-1}, 
    &\text{ for even } &j \in\JJ_k\cup\{\pm M_k\}, \\
    \bar q_{j}^{k} &
    = \tfrac1 2\(q_{\frac {j-1}2}^{k-1}+q_{\frac       {j+1}2}^{k-1}\) 
    &\text { for odd } &j \in\JJ_k.
  \end{alignat*}
  We define $\ze_k\defeq\frac16\ze_{k-1}$. For this choice, the   vertices of $\bar \gamma_k$ lie on the finer grid   $\(\e_k\Z,\ze_k\Z^{n-1}\)$. Then $\gamma_k$ is the Birkhoff map of   $\bar\gamma_k$ for the discretisation length $\e_k$.
\end{definition}
 
To simplify the notation, we write on each level $k$
\begin{equation*}
  \De_j^k\defeq\De_j^{+},\text{\quad which implies \quad}
  -\De_{j-1}^k=\De_j^{-}.
\end{equation*}

We now want to prove that the Birkhoff refinement laid out above will lead to a converging sequence of polygons if we start the iteration with properly chosen initial points. Specifically, let us first assume that the total energy $E$, and thus $Q$, is fixed, and let us choose a nonempty compact set $P\subset Q$. Then let $N$ be as in the proof of Proposition~\ref{prop:rhombus}, $N=3\(4+5\norm{\nabla e^h}\)$ with the norm taken over the set $P$.
\begin{definition}
  \label{def;admis}
  Any pair of endpoints $\(q_a,q_b\)$ is \emph{admissible} if they   meet the following two conditions, which depend on $N$:
  \begin{enumerate}
  \item $\ell = \frac{\abs{q_a-q_b}}2 \leq \ell_0$, with $\ell_0     \defeq\frac1{4N\sqrt{n-1}}$ as in Lemma~\ref {lem:par-nbhd};
  \item the set $\PPtwo$ defined in Lemma~\ref{lem:par-nbhd} is     contained in $P$.
  \end{enumerate}
\end{definition}

\begin{theorem}
  \label{theo:dq-est}
  For an admissible pair $\(q_a,q_b\)$, we consider the sequence of   graphs $\(x,f_k(x)\)$ with $x\in\(-\ell,\ell\)$ representing the   polygons $\ga_k$ obtained by the Birkhoff refinement. As   $\ga_k\in\PP$, the centred differences satisfy for all $k\ge0$ the   combined estimates
  \begin{gather}
    \max_{j\in\JJ_k\cup\{-M_k\}}\abs{\frac{\De_j^k}{\e_k}}\le 1,
    \label{eq:dq-est}\\
    \max_{j\in \JJ_k}\max_{i=2, \dots, n} \abs{e_i      
      \cdot\frac{\De_j^k-\De_{j-1}^k}{\e_k^2}} \le N
    .
    \label{eq:dq2-est}
  \end{gather}
\end{theorem}

\begin{proof}
Observe that the initial polygon $\ga_0$ is contained in $\PP$, since all its finite differences vanish.

By induction let $k=1,2, \ldots$. According to Definition~\ref{def:refine}, we embed $\ga_{k-1}$ into the finer grid of size $\e_k$, which is half the size of $\e_{k-1}$. The recursive definition of $\ze_k$ implies $\ze_k \le \e_k^{2+\al}$, with $\al=\frac{\log 6}{\log 2} -2 \in(0,1)$.

With respect to $\e_k$, the embedded polygons $\ga_{k-1}$ are in $\PPtwo$; this follows since all newly introduced vertices have a vanishing second difference quotient $\frac{\De_j^k-\De_{j-1}^k}{\e_k^2}$ at odd nodes $j\in\JJ_k$, whereas all quotients with even $j\in\JJ_k$ are doubled in size,
\begin{equation*}
  \frac{\De_j^k-\De_{j-1}^k}{\e_k^2}
  =2\frac{\De_{\frac j2}^{k-1}-\De_{{\frac j2}-1}^{k-1}}{\e_{k-1}^2}
  .
\end{equation*}
From Proposition~\ref{prop:par0inv}, it follows that the final polygon $\ga_k$ obtained from the Birkhoff map is in fact in the smaller set $\PP$. This is crucial for our argument, since it shows that one application of the Birkhoff map after a refinement retains the same discrete curvature bound as before.

This shows that every polygon $\ga_k$ with step-size $\e_k$ belongs to $\PP$. By Lemma~\ref{lem:par-nbhd} it obeys also a bound on the first difference quotients, which is uniform in $k$.
\end{proof}

\begin{corollary}
  \label{cor:conv-aa}
  For an admissible pair $\(q_a,q_b\)$, we consider a sequence of   graphs $\(x,f_k(x)\)$ with $x\in(-\ell,\ell)$ representing the   polygons $\ga_k$ obtained by the Birkhoff refinement. Then the   functions $f_k$ converge in $C^\be\((-\ell,\ell); \R^{n-1}\)$ to a   limit $f\in C^{0,1}\((-\ell,\ell); \R^{n-1}\)$.  We write   $\gamma=(x,f(x))$ for the limit graph.
\end{corollary}

\begin{proof}
  As $\ga_k\in C^{0,1}([0,1];S)$, we deduce the claimed convergence in   $C^\be([0,1];S)$ for any $\be\in(0,1)$ by Arzel\`a-Ascoli.
\end{proof}

Let us remark that the graphs of all polygons $\gamma_k$ and hence the limit $\gamma$ belong to the subset $\PN$ of $Q$ defined in Lemma~\ref{lem:par-nbhd}.

\subsection{Variational formulation for a geodesic}
\label{sec:Vari-form-geod}

Let us consider a geodesic which can be represented as a graph. In this subsection, we derive a weak formulation for its governing equation. Thus, we consider a curve represented as a graph $\ga(x)=(x,f(x))$, $x\in(-\ell,\ell)$. The length of this curve in the Jacobi metric $g^\Jac_{ij}(q)=e^{2h}(q)\delta_{ij}$ is given by
\begin{equation*}
  L[\ga]=\int_{-\ell}^{        
    \ell}\sqrt{g^\Jac_{ij}(\ga(x))\dot\ga^i\dot\ga^j}\d x
  =\int_{-\ell}^{\ell}e^{h(\ga(x))}\sqrt{1+\abs{\dedx f}^2}\d x.
\end{equation*}
For a \emph{orthogonal} perturbation $\ga_\e=(x,f+\e y)$, where $y$ has compact support in $\(-\ell,\ell\)$, we deduce via integration by parts
\begin{align*}
  \left.\frac d{d\e}\right|_{\e=0}L[{\ga_\e}]
  &=\int_{-\ell}^{\ell}  
  \left[y(x)\cdot\nabla_Ne^{h(\ga(x))}\sqrt{1+\abs{\dedx        
        f(x)}^2}+e^{h(\ga(x))}\frac{\dedx f(x)\cdot \dedx
      y(x)}{\sqrt{1+\abs{\dedx f(x)}^2}}\right] \d x
  .
\end{align*}
Thus, a graph $\ga(x)=(x,f(x))$ with $f\in H^1\(-\ell,\ell\)$ is stationary for the length functional $L_\ga$ if the variational derivative $\left.\frac d{d\e}\right|_{\e=0}L_{\ga_\e}$ vanishes for all functions $y\in H_0^1\(-\ell,\ell\)$.  We will use this weak formulation, rather than the more common one obtained by a further integration by parts.

\subsection{Characterisation of the limit of the Birkhoff refinement}
\label{sec:char-lim}

Using a smoother interpolation of the grid points representing the polygons of the Birkhoff refinement we obtain a better convergence and a smoother characterisation of its limit than the previous result of Corollary~\ref{cor:conv-aa}.

\begin{theorem}
  \label{theo:char-lim}
  For an admissible pair $\(q_a,q_b\)$, consider the sequence of   polygons $\ga_k=\(x,f_k(x)\)$ obtained by the Birkhoff refinement.   The functions $f_k$ converge in $H^1\(-\ell,\ell\)$ to a limit $f$.   Furthermore, $f\in C^{1,1} \(-\ell,\ell\)$ and the limit graph   $\ga(x)=(x,f(x))$ satisfies
  \begin{align}
    0
    &=\int_{-\ell}^{\ell}    
    \left[y(x)\cdot\nabla_Ne^{h(\ga(x))}\sqrt{1+\abs{\dedx      
          f(x)}^2}+e^{h(\ga(x))}\frac{\dedx f(x)\cdot \dedx
        y(x)}{\sqrt{1+\abs{\dedx f(x)}^2}}\right]\d x 
    . \label{eq:weak-geo}
  \end{align}
  for every function $y\in H_0^1\(-\ell,\ell\)$.
\end{theorem}

\begin{proof}
Analogously to the definition of $\Cbar(\e,\de)$ in~\eqref{eq:Ceed}, we define
\begin{equation*}
  \Cbar_j(\e_k,\ze_k \si e_i)\defeq
  \(\Lbar{\seg{q_{j-1}}{q_j+\ze_k \si e_i}}
  +\Lbar{\seg{q_j+\ze_k \si e_i}{q_{j+1}}}   \)
  -\(\Lbar{\seg{q_{j-1}}{q_j}} +
  \Lbar{\seg{q_j}{q_{j+1}}}\) . 
\end{equation*}
It is convenient to write for $j\in\JJ_k$ in analogy to~\eqref{eq:E-pm}
\begin{align*}
  q_{j+\frac12}^k&\defeq \frac12\(q_j^k+q_{j+1}^k\), \\
  X_{j+\frac12}^k&\defeq \frac12\(X_j^k+X_{j+1}^k\). 
\end{align*}
In further analogy to~\eqref{eq:G-pm}, let
\begin{align}
\G_j^k&\defeq 
  \frac{\frac{\De_j^k}{\e_k}}{\sqrt{1+\abs{\frac{\De_j^k}{\e_k}}^2}}
. \label{eq:Gj-pm}
\end{align}
We want to examine $\Cbar(\e,\de)$ of~\eqref{eq:Ceed} evaluated at a node $q_j$ for $\e=\e_k$ and $\de=\ze_k \si e_i$, where $\si\in\{\pm1\}$ and $i\in\{2,\ldots,n\}$.  Theorem~\ref{theo:dq-est} provides the necessary estimates for the difference quotients, so that
\begin{align*}
   \nabla_Ne^h(q_j^k)
    \frac{\({\frac{\De_j^k}{{\e_k}}}+{\frac{-\De_{j-1}^k}{-{\e_k}}}\)
          \({\frac{\De_j^k}{{\e_k}}}-{\frac{-\De_{j-1}^k}{-{\e_k}}}\)}
      {2\(\sqrt{1+\abs{\frac{\De_j^k}{{\e_k}}}^2}
        +\sqrt{1+\abs{\frac{-\De_{j-1}^k}{-{\e_k}}}^2}\)} = O({\e_k}),
\end{align*}
since $\frac{{\frac{\De_j^k}{{\e_k}}}+{\frac{\De_{j-1}^k}{{\e_k}}}} {\sqrt{1+\abs{\frac{\De_j^k}{{\e_k}}}^2}+   \sqrt{1+\abs{\frac{\De_{j-1}^k}{{\e_k}}}^2}} \le 1$ and $\frac{\De_j^k-\De_{j-1}^k}{\e_k} =O({\e_k})$ by~\eqref{eq:dq-est}--\eqref{eq:dq2-est}. Further, by Definition~\ref{def:refine},
\begin{align*}
  O\(\abs{\frac{\ze_k}{\e_k^2}}\) = O\(\abs{\e_k}^\al\). 
\end{align*}
Thus, we can write
\begin{equation*}
  \frac{\Cbar_j(\e_k,\ze_k \si    
    e_i)}{\si\e_k\abs{\ze_k}}+O(\abs{\e_k}^\al) 
  =e_i
  \cdot\left(
  -\frac{\Ek{j+\frac12} \G_j^k-\Ek{j-\frac12}\G_{j-1}^k}{\e_k}
  +
  \nabla_Ne^h\(q_j\)\sqrt{1+\abs{\frac{\De_j^k}{\e_k}}^2}
  \right)
  .
\end{equation*}
We identify the right hand side as a product of a function $w$ evaluated on the piecewise constant interpolation $\gapc$ times a function $z$ evaluated on the piecewise linear interpolation $\gapl$, that is, for all $x\in\(X_{j-\frac12}^k , X_{j+\frac12}^k\)$ there holds
\begin{multline}
  \label{eq:cj-est}
  \frac{\Cbar_j(\e_k,\ze_k \si
    e_i)}{\si\e_k\abs{\ze_k}}+O(\abs{\e_k}^\al) \\
  =e_i
  \cdot\left[
    D_{\e_k}\(e^h\(x,\gapc(x)\) 
    \G\(\dedx\gapl(x)\)\)+\DE\(x,\gapc(x)\)  \PG\(\dedx\gapl(x)\)
  \right]
  ,
\end{multline}
where $D_{\e_k}$ is the centred difference quotient
\begin{equation*}
  D_{\e}y(x)\defeq\frac{y(x+\frac\e2)-y(x-\frac\e2)}\e
\end{equation*}
and
\begin{equation*}
  \G(\xi) \defeq   \frac{\xi}{\sqrt{1+\abs{\xi}^2}}, \quad
  \PG(\xi) \defeq   \sqrt{1+\abs{\xi}^2}. 
\end{equation*}
We remark that the right-hand of~\eqref{eq:cj-est} side is globally defined for $x\in(-\ell,\ell)$, and piecewise constant. Given an arbitrary fixed function $y\in H_0^1\(-\ell,\ell\)$, we can multiply~\eqref{eq:cj-est} with $y$ and integrate by parts to obtain
\begin{align}
  \begin{split}
    & \sum_{j=-M+1}^{M-1}\int_{X_{j-\frac12}}^{X_{j+\frac12}}
    y(x)\left[\frac{\Cbar_j(\e_k,\ze_k \si        
        e_i)}{\si\e_k\abs{\ze_k}}+O(\abs{\e_k}^\al)\right] \d x  \\
    &=e_i
    \cdot\int_{-\ell+\frac {\e_k}2}^{\ell-\frac {\e_k}2}y(x)\left[
      D_{\e_k}\(e^h\(x,\gapc(x)\)       
      \G\(\dedx\gapl(x)\)\)+\DE\(x,\gapc(x)\)  \PG\(\dedx\gapl(x)\)
    \right] \d x  \\
  &=e_i
  \cdot\int_{-\ell+\frac {\e_k}2}^{\ell-\frac {\e_k}2}\left[
    - \(D_{\e_k}y(x)\) e^h\(x,\gapc(x)\)\G\(\dedx\gapl(x)\)
    +y(x)\DE\(x,\gapc(x)\)  \PG\(\dedx\gapl(x)\)
  \right] \d x
  . 
  \end{split}
  \label{eq:int-est}
\end{align}

We observe that the right-hand side does not depend on $\si$.  Now recall that stoppage of the algorithm on level $k$ implies that $\Cbar_j(\e_k,\ze_k \si e_i)\ge0$ for both choices of the sign of $\si=\pm1$.

Let us assume for the moment that piecewise constant interpolation $\gapc$ and the piecewise linear interpolation $\gapl$ converge strongly to the same limit $\ga$ as $k\to\infty$ in the sense that
\begin{align}
  \norm{\gapc-\ga}_{L^2}&\to0,  \label{eq:con-con}
  \intertext{and}
\norm{\dedx\gapl-\dedx \ga}_{L^2}&\to0.  \label{eq:con-lin}
\end{align}
Then the argument can be finished as follows.  Observe that $\Cbar_j(\e_k,\ze_k \si e_i)$ is non-negative for both choices of $\si=\pm1$, whereas the last line of~\eqref{eq:int-est} does not depend on the chosen $\si$ anymore. On each interval $\(X_{j-\frac12} , X_{j+\frac12}\)$, we first choose $\sigma$ to have the same sign as $\int_{X_{j-\frac12}}^{X_{j+\frac12}}y(x)\d x$. Then the sum on the left-hand side is non-negative. We pass to the limit $k\to\infty$ in~\eqref{eq:int-est} and find
\begin{multline}
  0 \le  
  \lim_{k\to\infty}\sum_{j=-M+1}^{M-1}\int_{X_{j-\frac12}}^{X_{j+\frac12}} 
  y(x) \frac{\Cbar_j(\e_k,\ze_k \si e_i)}{\si\e_k\abs{\ze_k}} \d x \\
  =e_i
  \cdot\int_{-\ell}^{\ell}\left[
    - \(\dedx y(x)\) e^h\(x,\ga(x)\)\G\(\dedx\ga(x)\)
    +y(x)\DE\(x,\ga(x)\)  \PG\(\dedx\ga(x)\)
  \right] \d x
  ; 
\end{multline} 
here the convergence of the difference quotient $D_{\e_k}y(x)$ follows from~\cite[Lemma 7.24]{Gilbarg:01a}.

Similarly, we then choose $\sigma$ to have always the opposite sign and obtain the reversed inequality. Together this yields
\begin{equation*}
  0 = e_i
  \cdot\int_{-\ell}^{\ell}\left[
    - \(\dedx y(x)\) e^h\(x,\ga(x)\)\G\(\dedx\ga(x)\)
    +y(x)\DE\(x,\ga(x)\)  \PG\(\dedx\ga(x)\)
  \right] \d x
  . 
\end{equation*} 
As we can test all normal directions $e_i$, with $i=2,\ldots,n$, in this fashion, we obtain the vectorial identity
\begin{equation}
  \label{eq:weak-geodes}
  0 = \int_{-\ell}^{\ell}\left[
   - \(\dedx y(x)\) e^h\(x,\ga(x)\)\G\(\dedx\ga(x)\)
   +y(x)\DE\(x,\ga(x)\)  \PG\(\dedx\ga(x)\)
  \right] \d x
\end{equation} 
in $\R^{n-1}$. By substituting the definitions of $\G$ and $\PG$, we recover the claim.
\end{proof}

\subsubsection{Proof of~\protect{\eqref{eq:con-con}} and~\protect{\eqref{eq:con-lin}}}
\label{sec:Proof-prot-int}
The previously assumed convergence of the interpolants is established in the following arguments.

\begin{lemma}[Estimates for interpolants]
  \label{lem:fin-el}
  Given a function $\gapqo\in H^2$, let $\gapco$ be the piecewise   constant interpolation on an equidistant grid of size $\e$, and   similarly $\gaplo$ the piecewise linear interpolation. Then the   following error estimates hold
  \begin{align*}
    \norm{\gapco-\gapqo}_{L^2}&\le C\e^2\norm{\dezdx\gapqo}_{L^2}
    ,  
    \intertext{and}
    \norm{\dedx\gaplo-\dedx\gapqo}_{L^2}&\le  
    C\e\norm{\dezdx\gapqo}_{L^2}  
    .
  \end{align*}
\end{lemma}
\begin{proof}
This is a standard argument in Finite Elements, see for example~\cite[Theorem 0.8.7 and Section 4]{Brenner:08a}. 
\end{proof}

For a fixed refinement level $k$, let $q_j$ be the point set associated to $\ga_k$, the output of the Birkhoff map at level $k$. Now, we first construct a quadratic interpolation $\gapq$ of $\left\{q_j^k\right\}$ with the special property that its piecewise constant and piecewise linear interpolation coincide with the interpolations $\gapc$ and $\gapl$ introduced before. Specifically, in our situation the quadratic interpolation can be chosen in such a way that $\norm{\dezdx\gapq}_{L^2}$ is bounded independently of $k$.

Explicitly, to construct $\gapq$, any two neighbouring nodes $q_j^k=(X_j^k,Y_j^k)$ and $q_{j+1}^k=(X_{j+1}^k,Y_{j+1}^k)=(X_j^k+\e_k,Y_j^k+\De^k_j)$ are connected via two quadratic splines $s_\pm$ on $\(X_j^k,X_{j+\frac12}^k\)$ and $\(X_{j+\frac12}^k,X_{j+1}^k\)$, with matching conditions
\begin{alignat*}2
  s_-(X_j^k)&=Y_j^k, \quad &
  s'_-(X_j^k)&=b_-\defeq\frac{\De_j^k+\De_{j-1}^k}{2\e_k}, \\
  s_+(X_{j+1}^k)&=Y_{j+1}^k, \quad &
  s'_+(X_{j+1}^k)&=b_+\defeq\frac{\De_{j+1}^k+\De_{j}^k}{2\e_k}, \\
  s_-(X_{j+\frac12}^k)&=s_+(X_{j+\frac12}^k), \quad &
  s'_-(X_{j+\frac12}^k)&=s'_+(X_{j+\frac12}^k).
  \intertext{On the boundary, we vary the definition slightly for    
    $j=-M_k$ and $j=M_k-1$ respectively,}
  s_-(X_{-M_k}^k)&=Y_{-M_k}^k, \quad &
  s'_-(X_{-M_k}^k)&=b_-\defeq\frac{\De_j^k}{\e_k}
  \intertext{and}
  s_+(X_{M_k}^k)&=Y_{M_k}^k, \quad &
  s'_+(X_{M_k}^k)&=b_+\defeq\frac{\De_{M_k-1}^k}{\e_k}
  .
\end{alignat*}

We recall that the polygons $\ga_k=\(x,f_k(x)\)$ obtained by the Birkhoff refinement satisfy the finite difference estimates~\eqref{eq:dq-est}--\eqref{eq:dq2-est}.  For the chosen assignments $b_\pm$ of the first derivatives at the nodes, the piecewise quadratic interpolation $\gapq$ is continuously differentiable throughout.  Further, it can be computed that the bound on the second difference quotients implies a bound
\begin{align}
  \label{eq:d2-bound}
  \norm{\dezdx\gapq}_{L^\infty}\le C,
\end{align}
which holds uniformly in $k$. Hence the sequence $\left\{\gapq\right\}$ is uniformly bounded in $C^{1,1}$, and we infer convergence in $C^{1,\be}$ to $\ga\in C^{1,1}$ by Arzel\`a-Ascoli. This implies that
\begin{align}
  \norm{\gapq-\ga}_{L^2}&\to0,  \label{eq:q-con}
  \intertext{and}
\norm{\dedx\gapq-\dedx \ga}_{L^2}&\to0  \label{eq:q-one-con}
\end{align}
as $k\to\infty$. Choosing $\gapq$ as the function $\gapqo$ in Lemma~\ref{lem:fin-el}, we infer for all $k$ that
\begin{align*}
  \norm{\gapc-\ga}_{L^2}&\le \norm{\gapc-\gapq}_{L^2} + \norm{\gapq-\ga}_{L^2} \\ &\le C\e_k^2\norm{\dezdx\gapq}_{L^2} + \norm{\gapq-\ga}_{L^2}
  \intertext{and}
\norm{\dedx\gapl-\dedx \ga}_{L^2}& \le \norm{\dedx\gapc-\dedx\gapq}_{L^2} + \norm{\dedx\gapq-\dedx\ga}_{L^2} \\& \le C\e_k\norm{\dezdx\gapq}_{L^2} + \norm{\dedx\gapq-\dedx\ga}_{L^2} 
. 
\end{align*}
All expressions tend to $0$ by~\eqref{eq:d2-bound}--\eqref{eq:q-one-con}. This proves the previously assumed claims~\eqref{eq:con-con} and~\eqref{eq:con-lin}.

\section{Numerical investigations}
\label{sec:Numer-investigations}

Here we present experimental convergence rates and estimates for the computational effort for a simple benchmark problem.  First, we obtain an explicit solution for a special metric $g^\Jac$, which will serve as comparison for the numerical approximations on different discretisation levels $\e_k$.

\subsection{Special analytic solution}

From the weak formulation of the geodesic equation for graphs~\eqref{eq:weak-geo}, we obtain by integration by parts the following strong version
\begin{multline*}
\nonumber
  \nabla_Ne^h\sqrt{1+\a{f'}^2}
 -\(\frac d{dx}e^h+f'\cdot\nabla_Ne^h\sqrt{1+\a{f'}^2}\)\frac{f'}{\sqrt{1+\a{f'}^2}}
=
 e^h\frac{(1+\a{f'}^2)I-f'f'^T}{\sqrt{1+\a{f'}^2}^3}f''
 ,
\end{multline*} 
which has to satisfied by a geodesic connecting the two points $(\pm\tfrac\ell2,0)$ being a graph $\gamma(x)=(x,f(x))$, $f(\pm\frac\ell2)=0$.

If we assume further that $h$ does not depend on $x$, the last equation simplifies to
\begin{align*}
0 =&
  \frac{\Big((1+\a{f'}^2)I-f'f'^T\Big)}{\sqrt{1+\a{f'}^2}^3}
 \(\nabla_N{h}\(1+\a{f'}^2\)-f''\)
 .
\end{align*}
That is, we need to solve the simpler equation
$$f''=\nabla_N{h}\(1+\a{f'}^2\).$$
Let us restrict ourselves to linear $h$, not depending on the `horizontal' co-ordinate $x$, that is, $h=h(y)=-\al n^Ty$ with $\al>0$ and $n\in\R^{n-1}$ a unit vector, $\a n=1$.  Then $\nabla_N{h}=-\al n$, hence, $f''=-\al\(1+\a{f'}^2\)n$.  With the ansatz $f'=\phi'n$, the equation can be rewritten as
\begin{equation*}
  \phi''n=-\al n\(1+\a{\phi'}^2\),
\end{equation*}
hence we are left with the scalar equation
\begin{equation*}
  \phi''=-\al\(1+\a{\phi'}^2\)
\end{equation*}
for $\phi'$ alone. Its solution is $\phi'(x)=\tan(C- \al x)$.  Thus, integrating once again we deduce $\phi(x)=\frac1\al\ln\cos(-C+\al x)+D$.  Matching the boundary conditions, we find that the geodesic connecting the two points $(\pm\tfrac\ell2,0)$ is
\begin{equation*}
  \ga(x)=(x,f(x))
, \quad
x\in(-\tfrac \ell2,\tfrac \ell2),\
f(x)=n\phi(x)
\end{equation*}
with
\begin{equation*}
  \phi(x)=\frac1\al\ln\frac{\cos\al x}{\cos\al\tfrac \ell2}.
\end{equation*}
We remark that this representation requires $\al\tfrac \ell2<\frac\pi2$, or equivalently $\al<\frac\pi\ell$.  Note that for $\al\to0$, the metric approaches the constant Euclidean metric $h\equiv0$, and we obtain the straight line segment $(x,f\equiv0)$ in the limit.

For the other extreme, we remark that if $\al\to\frac\pi\ell$, the geodesic converges to the set of two parallel lines $\{(\pm\tfrac \ell2,y) \bigm| y\ge0\}$ connected at infinity.  In fact, this set of constant and finite length will be the minimising configuration for all values of $\al\ge\frac\pi\ell$.

\subsection{Computational effort} 
We use the analytic solution found above as a benchmark test for convergence of the method.  Let
\begin{equation*}
  g^\Jac_{ij}(x,y)\defeq e^{2h(y)}\de_{ij}
  \quad\text{with}\quad
  h(y)=-\al y
  .
\end{equation*}

We now present computations in $\R^2$ with $\ell=2$, so connecting $(\pm1,0)$.  Let us consider three different metrics by choosing three values $\al_1=0.65$, $\al_2=0.9$, and $\al_3=1.1$, all less than $\frac\pi2$.  Larger values of $\alpha$ emphasise the difference of $g$ compared to the flat Euclidean metric, hence one expects a larger curvature in the geodesic, which bends further away from the horizontal straight connection between the boundary points $(\pm1,0)$. We also explore how the performance of the algorithm depends on this geometrical feature.

\begin{figure}
  \includegraphics[width=0.99\textwidth,viewport=55 200 545 600,clip=true]{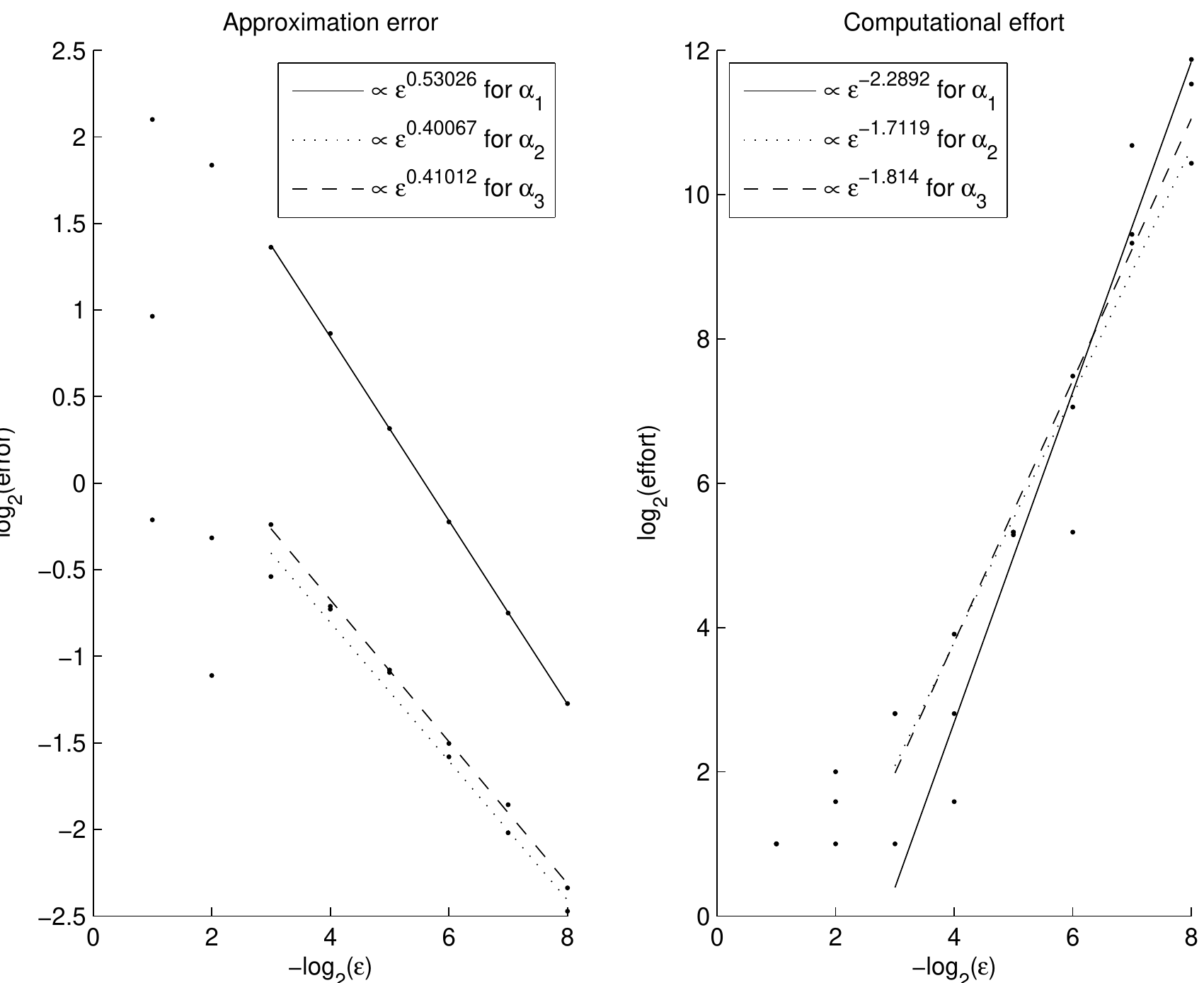}
\caption{\label{fig:exp-metric-log} Exponential metric, $n=2$:   $k=1,\ldots,8$. Left panel: The rate of the error proportional to   ${\e}^{m}$. Right panel: The effort grows approximately with the   rate $\e^{-2}$.}
\end{figure} 
We begin by presenting numerical statistics of the geodesic computation.  The plots in Figure~\ref{fig:exp-metric-log} show the computational error of the polygonal approximation compared to the explicit geodesic (measured in the $L^2$ norm) in the left panel and the computational effort, measured in the number of affirmative Birkhoff steps, in the right panel. 

The error decreases as $\e_k=2^{-k}$ decreases for $k=1,\ldots,8$.  In fact, the calculation shows the error is proportional to ${\e}^{m}$, where the exponent $m$ varies from 0.40 to 0.53, depending on the choice of $\al$ in the metric.

The effort increases as $\e$ decreases, and the simulation shows that the effort turns out to be (inversely) proportional to $1/{\e}^{m}$, where $m$ ranges from 1.71 to 2.2, as $\alpha$ is taken from $\al_1=0.65$, $\al_2=0.9$, $\al_3=1.1$.

Hence, for geodesics of larger curvature variation, the Birkhoff procedure starts closer but converges with smaller rate in $\e$, whereas the effort increases with largely uniform rate in $\e$ of about $1/\e^{2}$.

Figure~\ref{fig:exp-metric} shows how the polygons on the discrete level $\e_k$ approach the exact solution
\begin{equation*}
  \ga(x)=\(x,\frac1\al\ln\frac{\cos\al x}{\cos\al}\)
, \quad
x\in(-1,1)
.
\end{equation*}
We show the approximation for the three different values for $\alpha$ simultaneously in one graph. In each case, the solid line is the exact geodesic, the dotted line is the polygonal approximation for $e_3=1/8$, with the thicker dots indicating the stencil points. The intermediate polygon with densely distributed stencils is the approximation for $e_7=1/128$.

\begin{figure}
  \includegraphics[width=0.99\textwidth,viewport=55 200 545 600,clip=true]{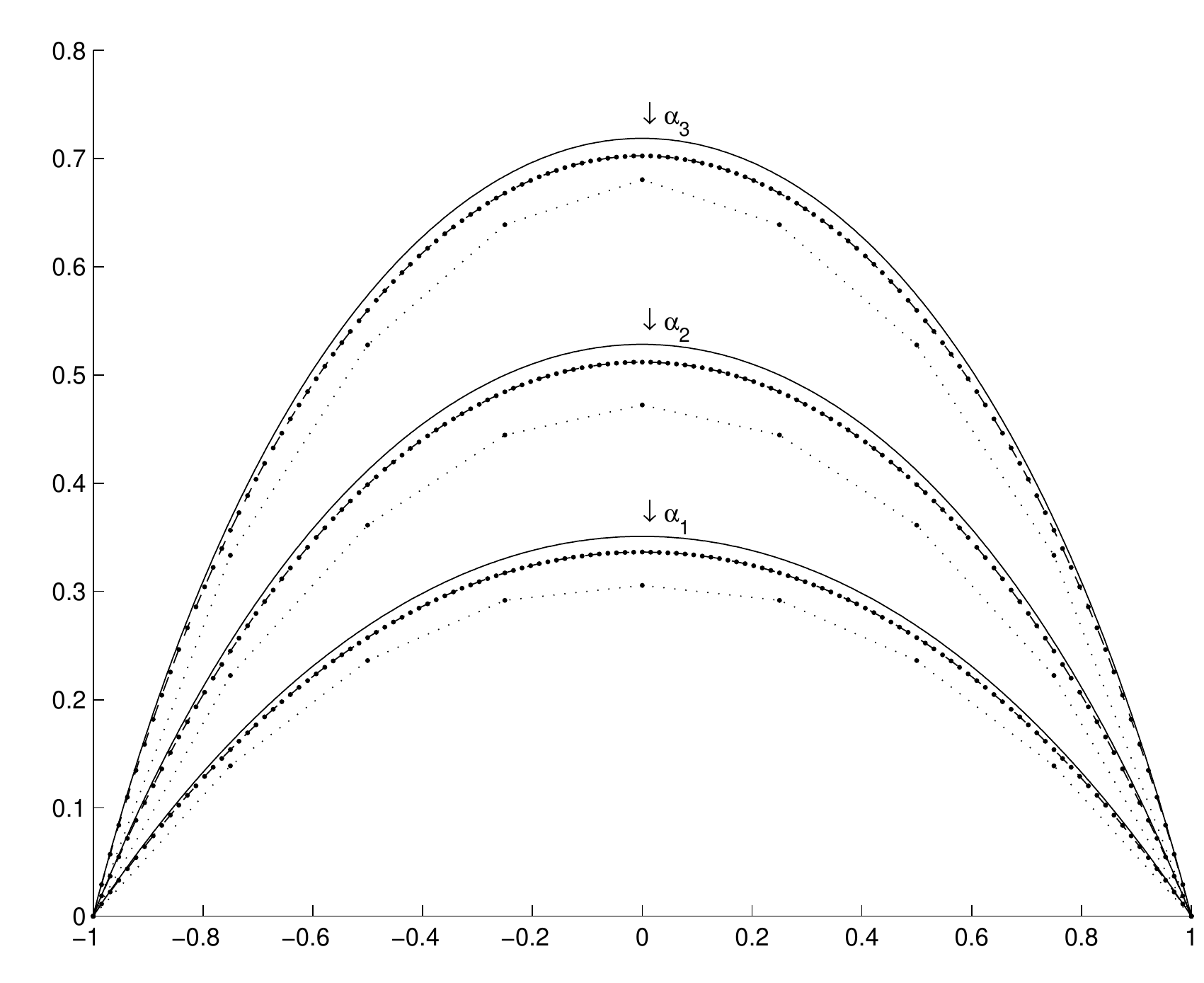}
\caption{\label{fig:exp-metric} Exponential metric, $n=2$: approximate   solutions compared to the analytic solution (solid line).}
        \end{figure}

\begin{acknowledgments}
  J.~Z.~gratefully acknowledges the financial support of the EPSRC   through an Advanced Research Fellowship (GR/S99037/1), grant   EP/K027743/1 and the Leverhulme Trust (RPG-2013-261). Both authors   are grateful to Michael Ortiz for pointing out the relevance of the   Jacobi principle. They benefited from helpful discussions during the   first annual meeting of the EPSRC network ``Mathematical Challenges   of Molecular Dynamics: A Chemo-Mathematical Forum'' (EP/F03685X/1)   and thank the anonymous reviewers for helpful comments.
\end{acknowledgments}

\def\cprime{$'$} \def\cprime{$'$} \def\cprime{$'$}   \def\polhk#1{\setbox0=\hbox{#1}{\ooalign{\hidewidth   \lower1.5ex\hbox{`}\hidewidth\crcr\unhbox0}}} \def\cprime{$'$}   \def\cprime{$'$} 


\end{document}